\documentclass[11pt]{article}

\usepackage[margin=1in]{geometry}
\usepackage{amsmath,amssymb,amsthm,mathtools,bm}
\usepackage{enumitem}
\usepackage{booktabs,tabularx}
\usepackage{microtype}
\usepackage{xcolor}
\usepackage[colorlinks=true,linkcolor=blue!55!black,citecolor=blue!55!black,urlcolor=blue!55!black]{hyperref}

\newtheorem{theorem}{Theorem}[section]
\newtheorem{proposition}[theorem]{Proposition}
\newtheorem{lemma}[theorem]{Lemma}

\theoremstyle{definition}

\newtheorem{remark}[theorem]{Remark}

\newcommand{\R}{\mathbb R}
\newcommand{\N}{\mathbb N}
\newcommand{\E}{\mathbb E}
\newcommand{\Pp}{\mathbb P}
\newcommand{\cN}{\mathcal N}
\newcommand{\cP}{\mathcal P}
\newcommand{\cQ}{\mathcal Q}
\newcommand{\cH}{\mathcal H}
\newcommand{\cS}{\mathcal S}
\newcommand{\Var}{\operatorname{Var}}
\newcommand{\Cov}{\operatorname{Cov}}
\newcommand{\tr}{\operatorname{tr}}
\newcommand{\Aff}{\operatorname{Aff}}
\newcommand{\TV}{\operatorname{TV}}
\newcommand{\NonG}{\operatorname{NG}}
\newcommand{\indep}{\mathrel{\perp\!\!\!\perp}}

\newcommand{\dirf}{\rightarrow}
\newcommand{\dirr}{\leftarrow}
\newcommand{\dd}{\,\mathrm d}
\newcommand{\one}{\mathbf 1}
\newcommand{\eps}{\varepsilon}
\newcommand{\He}{\operatorname{He}}
\DeclarePairedDelimiter{\norm}{\lVert}{\rVert}
\DeclarePairedDelimiter{\abs}{\lvert}{\rvert}
\DeclarePairedDelimiter{\ip}{\langle}{\rangle}
\newcolumntype{L}[1]{>{\raggedright\arraybackslash}p{#1}}
\newcolumntype{Y}{>{\raggedright\arraybackslash}X}

\title{How Many Samples Are Needed to Determine Causal Direction?
Sharp Minimax Bounds for Bivariate LiNGAM}
\author{Jikai Jin}
\date{August 2026}

\begin{document}
\maketitle

\begin{abstract}
We study how many observations are needed to determine the causal direction
between two linearly related variables.  Classical LiNGAM theory shows that
independent non-Gaussian disturbances identify the direction, but does not
quantify the difficulty when the causal effect is weak or the disturbances
are nearly Gaussian.  Let $\beta$ bound the absolute structural coefficient
from below, let $\nu$ measure each standardized disturbance's distance from
Gaussianity, and let the disturbance scales lie in
$[\underline\sigma,\overline\sigma]$.  We prove the sharp local minimax law
\[
 N_2^\star(\beta,\nu,\delta)
 \asymp
 \frac{\log(1/\delta)}
 {d_\beta^2+\beta^2\nu^2},
 \qquad
 d_\beta=
 \left[\beta^2-
 \left(1-\frac{\underline\sigma^2}{\overline\sigma^2}\right)\right]_+.
\]
Previous theory established population identifiability or assumed a fixed
separation between the two directions. By contrast, we establish the sharp sample complexity as a joint function of edge
strength, distance from Gaussianity, and scale uncertainty, and characterize
when identification comes from non-Gaussian dependence or from covariance
alone. The proof was independently generated with GPT-5.6 Sol in Codex's Ultra mode during a two-hour session. The human author supplied the prompt and was responsible only forchecking the proof and revising and polishing the manuscript.
\end{abstract}

\tableofcontents

\section{Introduction}

When two quantities move together, which one causes the other?  Even for two
variables joined by a linear relation, the observational data support two
regressions,
\[
 Y=aX+\varepsilon
 \qquad\text{or}\qquad
 X=bY+\widetilde\varepsilon.
\]
Ordinary least squares makes the residual uncorrelated with the regressor in
either direction.  The linear non-Gaussian acyclic model (LiNGAM)
distinguishes the directions by asking for independence: in the causal
direction the regressor is independent of the structural disturbance,
whereas in the reverse direction the regressor and residual are both
mixtures of the original disturbances
\cite{ShimizuEtAl2006,ShimizuEtAl2011,HyvarinenSmith2013}.

This asymmetry disappears at the Gaussian boundary.  For an independent
standardized Gaussian pair, every orthogonal rotation is again independent,
so both OLS directions pass the independence test.  A classical
characterization theorem shows that, when both rotation coefficients are
nonzero, independence of the rotated coordinates forces both standardized
sources to be Gaussian \cite{KaganLinnikRao1973}.  Error-scale restrictions
provide a second route to identification: under equal error variances, for
example, the covariance matrix can determine the direction even in the Gaussian limit
\cite{PetersBuhlmann2014,LohBuhlmann2014}.

Classical LiNGAM theory therefore explains when the direction is identifiable
at the population level, but not how the required sample size diverges near
the boundary.  Stability estimates for Gaussian characterization and local
analyses along prescribed near-Gaussian paths quantify related phenomena
\cite{Gabovich1974,Gabovich1981,SokolMaathuisFalkeborg2014}, while uniform
recovery fails without a positive separation
\cite{GeninMayoWilson2024}.  Recent finite-sample LiNGAM bounds instead
depend on a fixed residual-dependence or population-objective gap
\cite{OhHanPark2025,LaplanteAmbroiseHumbert2026}.  We derive this gap
uniformly from the edge strength, source non-Gaussianity, and admissible error
scales, and prove the corresponding sharp minimax law.

Two signals determine the difficulty.  Let $\beta$ be the smallest allowed
absolute edge coefficient and let $\nu$ denote the prescribed
non-Gaussianity margin of the standardized sources.  Fitting the regression
backward rotates the two independent sources.  The rotation mixes them by
order $\beta$; uniformly over sources satisfying the margin $\nu$, the
resulting dependence signal is bounded below by a constant multiple of
$\beta\nu$.  The error scales contribute a separate covariance signal
\[
 d_\beta:=\bigl[\beta^2-(1-\rho)\bigr]_+,
 \qquad
 \rho:=\frac{\underline\sigma^2}{\overline\sigma^2}.
\]
Here $\rho$ is the ratio of the smallest to the largest allowed noise
variance.  The signal $d_\beta$ vanishes exactly when the forward and reverse
covariance classes overlap.

For fixed tail, scale, and coefficient bounds, and for sufficiently small
$\beta$ and $\nu$, the number of observations needed to recover the direction
with error probability at most $\delta$ is
\[
 N_2^\star(\beta,\nu,\delta)
 \asymp
 \frac{\log(1/\delta)}{d_\beta^2+\beta^2\nu^2}.
\]
A signal of size $s$ requires order $s^{-2}$ observations, while confidence
$1-\delta$ contributes the factor $\log(1/\delta)$.  The test uses the
stronger of the covariance signal $d_\beta$ and the non-Gaussian signal
$\beta\nu$; their squared maximum is equivalent, up to constants, to the
sum in the denominator.  When the covariance classes overlap, the rate is
governed entirely by $\beta\nu$.  Under equal error variances,
$d_\beta=\beta^2$, so covariance remains informative as the sources approach
Gaussianity.

We define the source margin by a Gaussian-weighted distance between
characteristic functions, and the result holds uniformly over a
nonparametric sub-Gaussian class that includes symmetric sources, sources
whose first four moments match those of a Gaussian, and laws with atoms.  The
key analytic result is a quantitative
rotation theorem showing uniformly that the wrong-direction dependence is
at least a constant multiple of $\beta\nu$.  A test that uses the covariance
comparison when $d_\beta$ is larger and a robust independence score otherwise
attains the displayed rate; a matching forward--reverse construction shows
that no test can improve the resulting exponent.  This bivariate law is a
first step toward expressing the sample complexity of general LiNGAM
directly in terms of edge strengths, source non-Gaussianity, and error-scale
uncertainty.

\section{Literature review and the precise open gap}
\label{sec:literature}

Before formalizing the statistical experiment, we locate the result relative
to the identification, ICA, and finite-sample LiNGAM literature and isolate
the margin-dependent gap addressed here.

\paragraph{Population identifiability and algorithms.}
Shimizu, Hoyer, Hyv\"arinen and Kerminen introduced LiNGAM and used ICA to
identify a linear acyclic SEM with independent non-Gaussian disturbances
\cite{ShimizuEtAl2006}.  Their finite-sample discussion is algorithmic and
empirical, not a uniform error bound.  DirectLiNGAM later replaced the ICA
search with a sequential exogenous-variable procedure; Lemma~1 and the main
correctness claim are explicitly stated for infinite sample size
\cite{ShimizuEtAl2011}.
Hyv\"arinen and Smith studied the bivariate direction directly through
likelihood ratios and skewness/kurtosis approximations
\cite[Theorems~1--2]{HyvarinenSmith2013}.  Their standardized pairwise model
is the closest classical formulation to ours, but its cumulant criteria do
not cover sources for which both skewness and excess kurtosis vanish.  These
works establish population criteria and useful algorithms, but do not provide
a uniform finite-sample minimax law as both the edge and non-Gaussianity
vanish.

The qualitative identification step ultimately rests on the
Darmois--Skitovich characterization: if two nondegenerate linear forms of
independent variables are independent, every source appearing in both forms
is Gaussian \cite{KaganLinnikRao1973,FeldmanGraczyk2010}.  Stability versions
of characterization theorems were developed by Gabovich
\cite{Gabovich1974,Gabovich1981}, but those bounds use different probability
metrics and do not identify the sharp testing exponent for the composite
LiNGAM classes studied here.

\paragraph{Why an explicit margin is necessary.}
Sokol, Maathuis and Falkeborg quantified near-Gaussian ICA along a particular
contaminated-Gaussian path.  They found the local $n^{-1/2}$ contamination
boundary and noted consequences for linear causal models
\cite[Theorems~4.3--4.5]{SokolMaathuisFalkeborg2014}.  Their result is
asymptotic, tied to a specified contamination family, and concerns
mixing-matrix identifiability; their conclusion explicitly lists
finite-sample rather than limiting bounds as an open direction.  Genin and
Mayo-Wilson proved that uniform consistency is impossible on unseparated
LiNGAM model classes; their two-variable construction sends an edge
coefficient to zero, and their journal treatment formalizes the obstruction
in Theorem~9 \cite{GeninMayoWilson2024}.  That impossibility is exactly why a
quantitative margin such as $\NonG_w\ge\nu$ and a fixed lower edge bound are
logically indispensable: pointwise non-Gaussianity and a merely nonzero edge
are not uniform statistical assumptions.

\paragraph{Related quantitative ICA and high-dimensional LiNGAM results.}
Auddy and Yuan obtained minimax and computational results for
high-dimensional ICA.  Their model class imposes an $(8+\epsilon)$-moment
bound and a fixed two-sided lower and upper bound on every source excess
kurtosis, and their Theorem~2.1 targets mixing-direction estimation
\cite{AuddyYuan2025}.  This is not causal direction over all sub-Gaussian laws
separated from Gaussian in characteristic-function distance.  Oh, Han and
Park derived a high-dimensional LiNGAM structure-recovery rate of order
$d_{\rm in}\log(p/d_{\rm in})$ \cite{OhHanPark2025}.  Their Assumption~4
introduces an oracle residual distance-covariance gap $\tau_1>0$, and
Theorem~5 gives a tail containing
$\exp\{-c\epsilon_0^4n/\lambda^{36}\}$ for
$\epsilon_0<\tau_1/2$. Their Corollary~7 calls the dimension dependence optimal
when $\tau_1$ and the covariance conditioning constant $\lambda$ are fixed. However,
its lower bound does not track the deterioration as a source-level
non-Gaussianity margin tends to zero.  Thus the headline dimension rate does
not resolve the local problem here.

Recent work gives other useful results without closing this gap.  Li et
al. introduced a LiNGAM-specific kernel independence test and proved
asymptotic null and alternative laws \cite{LiEtAl2026}; it does not give a
two-direction minimax lower bound or a uniform rate as a source-level
non-Gaussianity margin tends to zero.  Laplante, Ambroise and Humbert proved population
identification and empirical convergence for a Wasserstein-to-Gaussian
criterion \cite{LaplanteAmbroiseHumbert2026}.  Their causal-order error bound
is expressed through an oracle objective gap $\Gamma$. Specifically, their Theorem~36 has a leading
$d^{5/4}n^{-1/4}$ uniform-objective term divided by $\Gamma$. However, no
lower bound relates $\Gamma$ sharply to weak edges and a source-level
near-Gaussian parameter.

Table~\ref{tab:literature} summarizes the logical distinctions comparing these related works to our central result.
\begin{table}[t]
\centering
\footnotesize
\caption{Closest primary results and the remaining ingredients.}
\vspace{5pt}
\label{tab:literature}
\begin{tabularx}{\textwidth}{@{}L{0.18\textwidth}Y L{0.18\textwidth}Y@{}}
\toprule
Work & Target and guarantee & Separation parameter & Relationship to our theorem\\
\midrule
Shimizu et al. (2006, 2011) & Population identification and algorithms & Qualitative non-Gaussianity & No uniform finite-sample rate\\
Hyv\"arinen--Smith (2013) & Pairwise likelihood/cumulant scores & Density or selected cumulants & No uniform power/lower bound\\
Sokol et al. (2014) & Near-Gaussian ICA local asymptotics & Particular contamination path & No finite-sample direction theorem\\
Genin--Mayo-Wilson (2024) & Uniform-recovery impossibility & No positive separation & Qualitative obstruction, no rate\\
Auddy--Yuan (2025) & Finite-sample/minimax ICA estimation & Fixed nonzero kurtosis & Different target and source class\\
Oh--Han--Park (2025) & Finite-sample graph recovery & Oracle $\tau_1$ & No primitive-gap lower bound\\
Li et al. (2026) & Asymptotic independence test & No uniform power margin & No finite-sample minimax theorem\\
Laplante et al. (2026) & Finite-sample causal-order upper bound & Oracle $\Gamma$ & No matching minimax lower bound; polynomial confidence conversion\\
This theorem & Finite-sample bivariate direction minimax law & $(\beta,\nu,d_\beta)$ & Matching upper and lower bounds\\
\bottomrule
\end{tabularx}
\end{table}

\paragraph{The equal-variance scenario.}
Gaussian SEMs become identifiable under equal error variances
\cite{PetersBuhlmann2014,LohBuhlmann2014}.  This explains why our answer need
not diverge as $\nu\downarrow0$ when
$\underline\sigma=\overline\sigma$: covariance already contains directional
information.  Prior equal-variance results and non-Gaussian LiNGAM results,
however, do not determine the sharp transition when the permitted scale
interval, $\beta$, and $\nu$ vary jointly.

\paragraph{Our contribution.}
We develop a sharp local complexity theory for LiNGAM direction in terms of
the primitive model parameters: edge strength, source non-Gaussianity, and
uncertainty in the error scales.  The central advance is a quantitative
modulus of LiNGAM identifiability.  We prove that fitting the model in the
wrong direction creates an observable dependence of order $\beta\nu$,
uniformly over a nonparametric sub-Gaussian source class.  This converts the
qualitative Darmois--Skitovich characterization into a finite-sample
separation stated directly in terms of the source distributions, without
introducing an unspecified population gap.  Exact covariance geometry
reveals a second signal, $d_\beta$, and the matching upper and lower bounds
show that the full testing exponent is
\[
 d_\beta^2+\beta^2\nu^2.
\]
The resulting phase law explains when direction is learned from
non-Gaussianity, when it is already visible in second moments, and how the
two mechanisms meet near the Gaussian boundary.  Since parent--child
orientation is the elementary local decision underlying a general LiNGAM,
this primitive-parameter characterization provides a first step toward a
sample-complexity theory for general LiNGAM structure recovery.

\section{Problem formulation and main theorem}
\label{sec:problem}

We now formalize the source class and the two directional experiments before
stating the minimax theorem.
Fix constants
\[
 K>K_{\rm G}:=\sqrt{8/3},\qquad
 0<\underline\sigma\le\overline\sigma<\infty,
 \qquad 0<\overline a<\infty.
\tag{3.1}\label{eq:fixedconstants}
\]
For a real random variable $Z$, use the Orlicz norm
\[
 \norm{Z}_{\psi_2}:=\inf\left\{r>0:
   \E\exp(Z^2/r^2)\le2\right\}.
\tag{3.2}\label{eq:psi2}
\]
Let
\[
 w(t):=\pi^{-1/2}e^{-t^2},\qquad g(t):=e^{-t^2/2},
\tag{3.3}\label{eq:w-g}
\]
and, for a centered variance-one $Z$ with characteristic function $f_Z$, set
\[
 \NonG_w(Z):=\left(\int_\R
   \abs{f_Z(t)-g(t)}^2w(t)\dd t\right)^{1/2}.
\tag{3.4}\label{eq:NG}
\]
For $\nu>0$, define the source class
\[
 \cQ(K,\nu):=\left\{\mathcal L(Z):
   \E Z=0,\ \E Z^2=1,\ \norm Z_{\psi_2}\le K,
   \ \NonG_w(Z)\ge\nu\right\}.
\tag{3.5}\label{eq:Q}
\]

For $0<\beta\le\overline a$, the forward class $\cP_{\dirf}(\beta,\nu)$
consists of the laws of
\[
 X=\sigma_1Z_1,
 \qquad Y=aX+\sigma_2Z_2,
\tag{3.6}\label{eq:forward}
\]
where
\[
 \beta\le\abs a\le\overline a,\qquad
 \underline\sigma\le\sigma_1,\sigma_2\le\overline\sigma,
 \qquad Z_1\indep Z_2,\quad
 \mathcal L(Z_j)\in\cQ(K,\nu).
\tag{3.7}\label{eq:forwardconstraints}
\]
The reverse class $\cP_{\dirr}(\beta,\nu)$ consists of the laws of
\[
 Y=\tau_2U_2,
 \qquad X=bY+\tau_1U_1,
\tag{3.8}\label{eq:reverse}
\]
under the analogous conditions
\[
 \beta\le\abs b\le\overline a,\qquad
 \underline\sigma\le\tau_1,\tau_2\le\overline\sigma,
 \qquad U_1\indep U_2,\quad
 \mathcal L(U_j)\in\cQ(K,\nu).
\tag{3.9}\label{eq:reverseconstraints}
\]
No density, symmetry, or nonvanishing-cumulant assumption is imposed.

A decision rule at sample size $n$ is a measurable, possibly randomized map
$\widehat d_n:(\R^2)^n\to\{\dirf,\dirr\}$.  Its worst directional error is
\[
 \mathcal R_n(\widehat d_n;\beta,\nu)
 :=\max\left\{
 \sup_{P\in\cP_{\dirf}}P^{\otimes n}(\widehat d_n=\dirr),
 \sup_{P\in\cP_{\dirr}}P^{\otimes n}(\widehat d_n=\dirf)
 \right\},
\tag{3.10}\label{eq:riskrule}
\]
where the dependence of the classes on $(\beta,\nu)$ is suppressed.  Put
\[
 \mathcal R_{2,n}^\star(\beta,\nu)
 :=\inf_{\widehat d_n}\mathcal R_n(\widehat d_n;\beta,\nu),
\qquad
 N_2^\star(\beta,\nu,\delta)
 :=\inf\{n\in\N:\mathcal R_{2,n}^\star(\beta,\nu)\le\delta\},
\tag{3.11}\label{eq:Nstar}
\]
with $\inf\varnothing=\infty$.

Write
\[
 L:=\underline\sigma^2,\qquad H:=\overline\sigma^2,
 \qquad \rho:=L/H,\qquad
 d_\beta:=\bigl[\beta^2-(1-\rho)\bigr]_+.
\tag{3.12}\label{eq:dbeta}
\]

\begin{theorem}[Sharp bivariate LiNGAM sample complexity]
\label{thm:main}
There exist constants
\[
 0<\beta_0\le\min\{1/2,\overline a\},\qquad
 \nu_0>0,\qquad \delta_0\in(0,1/2),
 \qquad 0<c<C<\infty,
\]
depending only on the fixed constants in \eqref{eq:fixedconstants}, such that
for every $0<\beta\le\beta_0$, $0<\nu\le\nu_0$, and
$0<\delta\le\delta_0$,
\[
 c\,\frac{\log(1/\delta)}{d_\beta^2+\beta^2\nu^2}
 \le N_2^\star(\beta,\nu,\delta)
 \le
 C\,\frac{\log(1/\delta)}{d_\beta^2+\beta^2\nu^2}.
\tag{3.13}\label{eq:mainrate}
\]
\end{theorem}

\begin{remark}[Regimes]
If $\underline\sigma<\overline\sigma$, then $d_\beta=0$ for all sufficiently
small $\beta$, and the rate is
$\asymp(\beta^2\nu^2)^{-1}\log(1/\delta)$.  If the scale is known and common,
$\underline\sigma=\overline\sigma$, then $d_\beta=\beta^2$ and
\[
 N_2^\star\asymp
 \frac{\log(1/\delta)}{\beta^2(\beta^2+\nu^2)}.
\]
Thus the rate saturates at $\beta^{-4}\log(1/\delta)$ when
$\nu\lesssim\beta$.  At a nontrivial covariance threshold, the transition is
$d_\beta\asymp\beta\nu$.
\end{remark}

\begin{remark}[Scope of the constants]
The theorem is local in $(\beta,\nu)$.  The constants and the radii
$\beta_0,\nu_0$ depend only on the fixed nuisance bounds, but the compactness
part of Theorem~\ref{thm:maxwell} does not provide closed-form dependence on
$K$ or on the lower bound imposed on the absolute cosine of the rotation.
Thus ``sharp'' refers to the scaling in $(\beta,\nu,\delta)$ for fixed
nuisance bounds, not to explicit numerical constants uniform over varying
source or rotation bounds.  Away from the local region, if
$d_\beta^2+\beta^2\nu^2$ is bounded below, the confidence dependence reduces
to the usual order $\log(1/\delta)$.
\end{remark}

\section{Proof outline}
\label{sec:architecture}

The proof of Theorem~\ref{thm:main} is assembled as follows.
\begin{enumerate}[label=\textbf{Step \arabic*.},leftmargin=*]
\item Proposition~\ref{prop:nonempty-ident} verifies that the source class is
nonempty and that the two directional classes are disjoint for $\nu>0$.
\item Proposition~\ref{prop:covariance} computes the exact covariance
separation and gives an estimator with exponent $nd_\beta^2$.
\item Theorem~\ref{thm:maxwell} proves that a rotation through angle $\theta$
of two admissible non-Gaussian sources has independence defect at least
$a_*\abs{\sin\theta\cos\theta}\nu$ in a third-order Gaussian Sobolev norm.
\item Proposition~\ref{prop:population-gap} converts the wrong OLS direction
into such a rotation and obtains a population score gap $c\beta\nu$.
Propositions~\ref{prop:robustmean}--\ref{prop:score-estimator} estimate that
gap at $n^{-1/2}$ scale, giving exponent $n\beta^2\nu^2$.
\item Proposition~\ref{prop:hardpair} constructs a forward law and a reverse
law whose squared Hellinger distance is at most
$C(d_\beta^2+\beta^2\nu^2)$.  Binary testing then gives the matching lower
bound.
\end{enumerate}
The upper bound selects the covariance or independence branch according to
which known separation is larger.  Since
$\max\{d_\beta^2,\beta^2\nu^2\}\asymp
d_\beta^2+\beta^2\nu^2$, the two branches join without an extra logarithm.

\paragraph{Notation map.}
The three analytic spaces below have different roles; they should not be
confused with the source class $\cQ(K,\nu)$.
\begin{center}
\small
\begin{tabularx}{0.96\textwidth}{@{}L{0.23\textwidth}Y@{}}
\toprule
Symbol & Role\\
\midrule
$\cQ(K,\nu)$ & Probability laws of standardized sub-Gaussian sources with
characteristic-function margin at least $\nu$.\\
$\mathsf S_r(\R^m)$ & Gaussian-weighted Sobolev space in frequency variables;
the independence score is an $\mathsf S_3$ norm.\\
$\mathfrak Q^r(\R^m)$ & Unweighted Shubin space obtained after multiplying an
$\mathsf S_r$ function by the Gaussian factor $e^{-\norm{x}_2^2/2}$.\\
$\cH_r$ & Hermite coefficient graph norm used only for linear coercivity.\\
$f_j,h_j,q_j$ & Source characteristic function, Gaussian difference
$h_j=f_j-g$, and Gaussian quotient $q_j=h_j/g$ (used only in the Maxwell
section).\\
$\pi_\eps,r_\eps$ & Product source density in the lower bound and its density
ratio relative to the bivariate Gaussian.\\
$(\alpha,\gamma)$ & Polynomial and derivative multiindices in the score
feature; unrelated to structural coefficients.\\
$d_\beta$ & Scale-induced covariance separation; the other population signal
is $\beta\nu$.\\
\bottomrule
\end{tabularx}
\end{center}

\section{Nonemptiness and qualitative identifiability}
\label{sec:nonempty}

Before studying quantitative separation, we verify that the statistical
experiment is nonvacuous.  We first record the qualitative characterization
used to rule out simultaneous forward and reverse representations.  We then
construct admissible non-Gaussian sources and apply that characterization to
prove that the two directional classes are disjoint.

\begin{lemma}[Bivariate Darmois--Skitovich theorem]
\label{lem:DS}
Let $V_1,V_2$ be independent nondegenerate real random variables.  If
$a_1V_1+a_2V_2$ and $b_1V_1+b_2V_2$ are independent and
$a_jb_j\ne0$ for $j=1,2$, then both $V_1$ and $V_2$ are Gaussian.
\end{lemma}

This is the classical two-source Darmois--Skitovich characterization
\cite{KaganLinnikRao1973,FeldmanGraczyk2010}.  We use it both here and in the
compactness argument of Section~\ref{sec:maxwell}; it requires neither
densities nor additional regularity.

\begin{proposition}[An admissible source path and exact direction identifiability]
\label{prop:nonempty-ident}
There is a constant $\nu_{\rm src}>0$, depending only on the fixed choice of
$K>K_{\rm G}$, such that $\cQ(K,\nu)\ne\varnothing$ for every
$0<\nu\le\nu_{\rm src}$.  Moreover,
\[
 \cP_{\dirf}(\beta,\nu)\cap\cP_{\dirr}(\beta,\nu)=\varnothing
 \qquad(\beta>0,\ \nu>0).
\tag{5.1}\label{eq:disjoint}
\]
\end{proposition}

\begin{proof}
Let $\phi$ be the standard Gaussian density and put
\[
 c_\circ:=\frac{e^{3/2}}2,
 \qquad h_\circ(x):=\sin x-c_\circ\sin(2x),
 \qquad M_\circ:=1+c_\circ,
 \qquad \eps_\circ:=\frac1{2M_\circ}.
\tag{5.2}\label{eq:sourceh}
\]
For $\abs\eps\le\eps_\circ$, define
\[
 p_\eps(x):=\phi(x)\{1+\eps h_\circ(x)\}.
\tag{5.3}\label{eq:sourcepath}
\]
Let $Z_\eps$ have density $p_\eps$.
Because $\abs{h_\circ}\le M_\circ$, the factor in braces lies in
$[1/2,3/2]$.  If $G\sim\cN(0,1)$, oddness of $h_\circ$ gives
\[
 \int_\R p_\eps(x)\dd x
 =1+\eps\E h_\circ(G)=1,
 \qquad
 \int_\R x^2p_\eps(x)\dd x
 =1+\eps\E[G^2h_\circ(G)]=1.
\]
Moreover,
\[
 \E[G\sin(kG)]=k e^{-k^2/2},
\]
and hence
\[
 \E[Gh_\circ(G)]
 =e^{-1/2}-c_\circ(2e^{-2})=0.
\]
Thus $p_\eps$ is centered and has variance one.  For every $r>\sqrt2$,
$x\mapsto e^{x^2/r^2}\phi(x)$ is even, so
\[
 \int e^{x^2/r^2}p_\eps(x)\dd x
 =\E e^{G^2/r^2}
 =\left(1-\frac2{r^2}\right)^{-1/2}.
\]
The last quantity is at most $2$ exactly when $r^2\ge8/3$; for
$r\le\sqrt2$, the integral is infinite because $p_\eps\ge\phi/2$.
Therefore \eqref{eq:psi2} gives
\[
 \norm{Z_\eps}_{\psi_2}=\norm G_{\psi_2}=\sqrt{8/3}=K_{\rm G}<K.
\tag{5.4}\label{eq:pathpsi}
\]

The characteristic function of $p_\eps$ is $g+\eps H_\circ$, where
\[
 H_\circ(t)=i\left[
 e^{-(t^2+1)/2}\sinh t
 -c_\circ e^{-(t^2+4)/2}\sinh(2t)
 \right].
\tag{5.5}\label{eq:Hcirc}
\]
Writing each sine as two complex exponentials gives this formula.  Moreover,
the choice of $c_\circ$ reduces it to
\[
 H_\circ(t)
 =i e^{-(t^2+1)/2}\sinh(t)\{1-\cosh(t)\},
\]
which is not identically zero.  Hence
\[
 A_\circ:=\left(\int\abs{H_\circ(t)}^2w(t)\dd t\right)^{1/2}>0,
 \qquad \NonG_w(Z_\eps)=A_\circ\abs\eps.
\tag{5.6}\label{eq:pathNG}
\]
Taking $\nu_{\rm src}:=A_\circ\eps_\circ$ and
$\eps=\nu/A_\circ$ proves nonemptiness.

For disjointness, suppose one law admits both representations.  In its
forward representation set
\[
 V:=\sigma_1Z_1,\qquad E:=\sigma_2Z_2,\qquad X=V,\qquad Y=aV+E.
\]
In any reverse representation, $Y\indep X-bY$; taking covariance and using
finite second moments forces
\[
 b=\frac{\Cov(X,Y)}{\Var(Y)}
 =\frac{a\sigma_1^2}{a^2\sigma_1^2+\sigma_2^2}.
\tag{5.7}\label{eq:reverseols}
\]
Thus $b\ne0$ and
\[
 1-ab=\frac{\sigma_2^2}{a^2\sigma_1^2+\sigma_2^2}>0.
\]
The independent reverse forms, expressed in the independent variables
$(V,E)$, are
\[
 Y=aV+E,
 \qquad X-bY=(1-ab)V-bE.
\tag{5.8}\label{eq:twoforms}
\]
All four coefficients in \eqref{eq:twoforms} are nonzero.
Lemma~\ref{lem:DS} therefore forces both $V$ and $E$ to be Gaussian.
After standardization, $Z_1,Z_2\sim\cN(0,1)$, contradicting
$\NonG_w(Z_j)\ge\nu>0$.  Lemma~\ref{lem:DS} applies to arbitrary probability laws,
so the argument does not assume densities or non-atomicity.
\end{proof}

\section{Exact covariance geometry}
\label{sec:covariance}

Section~\ref{sec:nonempty} established only that the two model classes are
disjoint.  We now ask how much of that separation is already visible in the
covariance matrix.  The answer is exact: covariance supplies the signal
$d_\beta$, and it supplies no uniform directional signal when $d_\beta=0$.

\begin{proposition}[Covariance overlap, separation, and testing]
\label{prop:covariance}
Assume $0<\beta\le\min\{1,\overline a\}$ and
$\cQ(K,\nu)\ne\varnothing$.  The sets of covariance matrices generated by the two
directional classes intersect if and only if
\[
 \beta^2\le1-\rho.
\tag{6.1}\label{eq:covintersect}
\]
For
\[
 \Delta(P):=\Var_P(Y)-\Var_P(X),
\]
one has
\[
 \inf_{P\in\cP_{\dirf}}\Delta(P)
 =H\{\beta^2-(1-\rho)\},
 \qquad
 \sup_{P\in\cP_{\dirr}}\Delta(P)
 =-H\{\beta^2-(1-\rho)\}.
\tag{6.2}\label{eq:Deltagap}
\]
In particular, when $d_\beta>0$, the two half-spaces are separated by a gap
$2Hd_\beta$.
Consequently, if $d_\beta>0$, the rule
\[
 \widehat d_n=
 \begin{cases}
  \dirf,&\widehat\Delta_n\ge0,\\
  \dirr,&\widehat\Delta_n<0,
 \end{cases}
 \qquad
 \widehat\Delta_n:=\frac1n\sum_{i=1}^n(Y_i^2-X_i^2)
\tag{6.3}\label{eq:hatDelta}
\]
has worst directional error at most
\[
 2e^{-cnd_\beta^2},
\tag{6.4}\label{eq:covrisk}
\]
where $c>0$ depends only on the fixed constants.
\end{proposition}

\begin{proof}
Write the forward source variances as $s=\sigma_1^2$, $t=\sigma_2^2$,
and the reverse source variances as $u=\tau_2^2$, $v=\tau_1^2$.  Equality of
a forward and reverse covariance matrix is equivalent to
\[
 u=a^2s+t,\qquad b=\frac{as}{u},\qquad v=\frac{st}{u}.
\tag{6.5}\label{eq:covequality}
\]
Indeed, these identities follow by equating respectively the $(2,2)$,
$(1,2)$, and $(1,1)$ entries; conversely they make every entry equal.
Set $\alpha=\abs a$ and $q=\abs b$.  Then \eqref{eq:covequality} becomes
\[
 u=\frac\alpha q s,\qquad
 v=(1-\alpha q)s,\qquad
 t=\frac\alpha q(1-\alpha q)s.
\tag{6.6}\label{eq:covratios}
\]
Relative to $s$, the four variances are
$1,r,c,rc$, where $r=\alpha/q$ and $c=1-\alpha q$.  Put
\[
 m_0:=\min\{1,r,c,rc\},\qquad
 M_0:=\max\{1,r,c,rc\}.
\]
A common scale $s>0$ places all four variances in $[L,H]$ precisely when
$[L/m_0,H/M_0]\ne\varnothing$, equivalently $M_0/m_0\le H/L$.  Since
$0<c<1$,
\[
 \frac{M_0}{m_0}
 =\frac{\max\{1,r\}}{c\min\{1,r\}}
 =\frac{\max\{r,r^{-1}\}}{c}.
\]
Thus feasibility is equivalent to
\[
 1-\alpha q\ge
 \rho\max\left\{\frac\alpha q,\frac q\alpha\right\}.
\tag{6.7}\label{eq:covfeasible}
\]
If $\alpha,q\ge\beta$, \eqref{eq:covfeasible} implies
$1-\beta^2\ge\rho$.  Conversely, when that inequality holds, take
\[
 \alpha=q=\beta,\qquad s=u=H,\qquad
 t=v=H(1-\beta^2)\in[L,H].
\]
This proves \eqref{eq:covintersect}.

Under a forward law,
\[
 \Delta=(a^2-1)s+t.
\]
For $\beta\le\abs a\le1$, the minimum over $(s,t)\in[L,H]^2$ is
$H(a^2-1)+L$, which is increasing in $a^2$ and is therefore minimized at
$\abs a=\beta$.  For $\abs a\ge1$, one has $\Delta\ge L$, while
$H(\beta^2-1)+L\le L$ because $\beta\le1$.  Hence
\[
 \inf\Delta=H\{\beta^2-(1-\rho)\}.
\]
Because the two covariance classes are obtained
from one another by swapping $X$ and $Y$, the reverse statement follows.  At
$d_\beta>0$, the boundary values $\abs a=\abs b=\beta$,
$s=u=H$, $t=v=L$ attain equality.  This proves \eqref{eq:Deltagap}.

Finally, define the sub-exponential Orlicz norm by
\[
 \norm{V}_{\psi_1}
 :=\inf\left\{r>0:\E\exp(\abs V/r)\le2\right\}.
\]
The model bounds give
$\norm X_{\psi_2}+\norm Y_{\psi_2}\le C_0$.  The implication
$\norm{U^2}_{\psi_1}\le C\norm U_{\psi_2}^2$, followed by the triangle
inequality and centering, yields
\[
 \norm{(Y^2-X^2)-\E(Y^2-X^2)}_{\psi_1}\le C_1.
\]
Bernstein's inequality gives
\[
 \Pp_P\bigl(\abs{\widehat\Delta_n-\Delta(P)}\ge Hd_\beta\bigr)
 \le2\exp\{-cn\min(d_\beta^2,d_\beta)\}
 =2e^{-cnd_\beta^2},
\]
where the last equality uses $0<d_\beta\le\beta^2\le1$ and absorbs $H$
into the fixed constant.
By \eqref{eq:Deltagap}, this event contains every sign error, proving
\eqref{eq:covrisk}.
\end{proof}

\section{A quantitative Maxwell--Darmois--Skitovich inequality}
\label{sec:maxwell}

Proposition~\ref{prop:covariance} settles the upper bound whenever covariance
provides the larger signal.  When the covariance classes overlap, direction
must instead be recovered from the dependence left by the wrong OLS
residual.  This section proves the quantitative rotation inequality needed
to lower-bound that dependence.

For $m\ge1$, let
\[
 W_m(t):=\pi^{-m/2}e^{-\norm{t}_2^2}
\]
and, for an integer $r\ge0$, define the Gaussian Sobolev norm
\[
 \norm F_{\mathsf S_r(\R^m)}^2
 :=\sum_{\abs\alpha+\abs\gamma\le r}
 \int_{\R^m}\abs{t^\alpha\partial^\gamma F(t)}^2W_m(t)\dd t.
\tag{7.1}\label{eq:Sr}
\]
For a centered variance-one pair $(S,T)$, write
\[
 \Phi(u,v):=\E e^{i(uS+vT)},\qquad
 D(u,v):=\Phi(u,v)-\Phi(u,0)\Phi(0,v).
\tag{7.2}\label{eq:defect}
\]
Thus $D=0$ if and only if $S\indep T$.

\begin{theorem}[Uniform Sobolev Maxwell inequality]
\label{thm:maxwell}
Fix $K<\infty$ and $c_*>0$.  There is
$a_*=a_*(K,c_*)>0$ such that the following holds.  Let $Z_1,Z_2$ be
independent, centered, variance-one variables with
$\norm{Z_j}_{\psi_2}\le K$, and let
\[
 S=cZ_1+sZ_2,\qquad T=-sZ_1+cZ_2,\qquad
 c^2+s^2=1,\qquad \abs c\ge c_*.
\tag{7.3}\label{eq:rotation}
\]
Then the defect \eqref{eq:defect} satisfies
\[
 \norm D_{\mathsf S_3(\R^2)}
 \ge a_*\abs{sc}
 \left\{\NonG_w(Z_1)^2+\NonG_w(Z_2)^2\right\}^{1/2}.
\tag{7.4}\label{eq:maxwell}
\]
\end{theorem}

We give all analytic details because the factor $\abs{sc}$ and uniformity
over arbitrary sub-Gaussian source laws are essential for the final rate.

\subsection{Proof strategy}

The standardized wrong-direction OLS pair derived in
Section~\ref{sec:popgap} is, up to swapping the sources and applying coordinate
reflections, an orthogonal rotation
\[
 S=cZ_1+sZ_2,
 \qquad
 T=-sZ_1+cZ_2.
\]
The analytic task in this section is to quantify the dependence
created by this rotation.  Write $D_\theta$ for the defect associated with
$c=\cos\theta$ and $s=\sin\theta$, and set
\[
 \eta
 :=
 \left\{
 \NonG_w(Z_1)^2+\NonG_w(Z_2)^2
 \right\}^{1/2}.
\]
The target scale is $\abs{sc}\eta$: the mixing contributes $\abs{sc}$, while
the departure of the sources from the rotation-invariant Gaussian pair
contributes $\eta$.

The origin of this product is clearest with a local path argument.  Specifically, let
$\varepsilon\mapsto f_{j,\varepsilon}=g+\varepsilon r_j$ be a local path of
characteristic functions of centered, variance-one sources through $g(t)=e^{-t^2/2}$, and
let $\theta$ vary near an unmixed angle
$\theta_0\in\pi\mathbb Z$.  Let $D_{\varepsilon,\theta}$ be the defect \eqref{eq:defect} of the pair
obtained by rotating independent sources with characteristic functions
$f_{1,\varepsilon}$ and $f_{2,\varepsilon}$ through the angle $\theta$.
Then
\[
 D_{0,\theta}=0
 \quad\text{for every }\theta,
 \qquad
 D_{\varepsilon,\theta_0}=0
 \quad\text{for every }\varepsilon.
\]
The first identity follows from rotational invariance of the Gaussian pair,
and the second from independence in the absence of mixing.  Consequently,
in the two-variable Taylor expansion in
$(\varepsilon,\theta-\theta_0)$, every term depending on only one of these
variables vanishes.  The first potentially nonzero term is
\[
 \varepsilon(\theta-\theta_0)\mathcal A(r_1,r_2),
 \qquad
 \mathcal A(r_1,r_2)
 :=\left.\partial_\varepsilon\partial_\theta
 D_{\varepsilon,\theta}\right|_{(0,\theta_0)}.
\]
Because every source along the path is standardized,
\[
 r_j(0)=r_j'(0)=r_j''(0)=0.
\]
The local problem is to show that $\mathcal A(r_1,r_2)\ne0$ whenever
$(r_1,r_2)\ne(0,0)$ subject to these three constraints.

To answer this question, we write $f_j=g+h_j$ and decompose the defect exactly as
\[
 D_\theta=D_{\mathrm{lin},\theta}+R_\theta,
\]
where the first term is linear in $(h_1,h_2)$ and every term in $R_\theta$
contains at least two perturbation factors.  The Hermite argument in
Lemma~\ref{lem:linearcoercivity} proves
\[
 \norm{D_{\mathrm{lin},\theta}}_{\mathsf S_3}
 \ge a_0\abs{sc}\eta.
\]
To prevent the full defect from being much smaller, it remains to show that
$R_\theta$ cannot cancel this linear signal.  Proposition~\ref{prop:nonlinear}
gives
\[
 \norm{R_\theta}_{\mathsf S_3}
 \le C\abs{sc}\eta^{4/3}.
\]
The Sobolev interpolation used below bounds the intermediate
$\mathsf S_4$ norm between the $L^2(w)$ and $\mathsf S_{12}$ norms.  Here the
higher power of $\eta$ comes from Sobolev interpolation and the factor
$\abs{sc}$ from integrating the angular derivative from an unmixed angle.
Consequently,
\[
 \norm{D_\theta}_{\mathsf S_3}
 \ge
 \abs{sc}\eta\bigl(a_0-C\eta^{1/3}\bigr)
 \ge \frac{a_0}{2}\abs{sc}\eta
\]
whenever $\eta\le\eta_{\mathrm{loc}}$, where
$\eta_{\mathrm{loc}}>0$ is chosen so that
$C\eta_{\mathrm{loc}}^{1/3}\le a_0/2$.  This proves the desired bound
near the Gaussian pair.

Note that this perturbative argument cannot cover the whole source class.
When $\eta\ge\eta_{\mathrm{loc}}$, the nonlinear remainder is no longer
guaranteed to be smaller than the linear term, so we instead use a
compactness argument.  The uniform sub-Gaussian bound makes the admissible
characteristic functions compact in $\mathsf S_4(\R)$, the source topology
used in the angular-continuity argument below.  In fact, the proof establishes
compactness in every fixed $\mathsf S_r(\R)$.

Because $D_\theta$ vanishes automatically when $sc=0$, for $sc\ne0$ we
divide out this known angular zero and set
\[
 \overline E(f_1,f_2,\theta):=\frac{D_\theta}{sc}.
\]
At an unmixed angle $\theta_0$, define its endpoint value as
$\partial_\theta D_{\theta_0}$.  Joint angular $C^1$ regularity ensures that
$D_\theta/(sc)$ converges to this endpoint value as
$\theta\to\theta_0$, and hence that $\overline E$ is continuous as an
$\mathsf S_3$-valued function.

This extension has no non-Gaussian zero.  At an interior angle, meaning
$sc\ne0$, a zero of
$\overline E$ would imply $D_\theta=0$, so the two rotated coordinates would
be independent; Darmois--Skitovich would then force both sources to be
Gaussian.  At an unmixed endpoint, a zero of
$\partial_\theta D_{\theta_0}$ gives the differential identity
\eqref{eq:Maxwellequation}; its only standardized sub-Gaussian
characteristic-function solution is the Gaussian pair.  Therefore
\[
 \frac{\norm{\overline E(f_1,f_2,\theta)}_{\mathsf S_3}}{\eta}
\]
is a continuous positive function on the compact region
$\eta\ge\eta_{\mathrm{loc}}$ and $\abs{\cos\theta}\ge c_*$.  It consequently
has a positive minimum there, completing the bound away from the Gaussian
pair and hence the proof on the full source class.

Gaussian conjugation and the product and angular-regularity lemmas provide
the norm estimates and endpoint continuity used above.  Once
Section~\ref{sec:popgap} gives $\abs{sc}\gtrsim\beta$ and
$\eta\gtrsim\nu$ for the wrong OLS direction, the resulting population gap is shown to be
at least a constant multiple of $\beta\nu$, producing the
$\beta^2\nu^2$ testing signal.

\subsection{Gaussian conjugation}

The proof strategy separates a local Hermite estimate from a nonlinear
remainder estimate.  Both are expressed in the weighted norm
$\mathsf S_r$.  We first remove that Gaussian weight so that ordinary
multiplication and differentiation estimates can be used in the later
product bounds.

For $m\ge1$, define the integer Shubin norm
\[
 \norm F_{\mathfrak Q^r(\R^m)}^2
 :=\sum_{\abs\alpha+\abs\gamma\le r}
 \norm{x^\alpha\partial^\gamma F}_{L^2(\dd x)}^2,
 \qquad \gamma_m(x):=e^{-\norm{x}_2^2/2}.
\tag{7.5}\label{eq:Shubin}
\]
The next lemma removes the Gaussian weight and replaces $\mathsf S_r$ by
equivalent Shubin and word\footnote{A word of length $k$ is an ordered composition
$A_1\cdots A_k$, where each $A_\ell$ is either a coordinate multiplication
operator $X_j$, defined by $(X_jF)(x)=x_jF(x)$, or a differentiation operator
$\partial_j$.  The empty word has length zero and denotes the identity
operator; $\abs A$ denotes the length of the word $A$.} norms.

\begin{lemma}[Gaussian conjugation]
\label{lem:conjugation}
Multiplication by $\gamma_m$ is an isomorphism from $\mathsf S_r(\R^m)$ onto
$\mathfrak Q^r(\R^m)$.  More precisely, writing $G=\gamma_mF$,
\[
 \norm F_{\mathsf S_r}
 \asymp_{m,r}\norm G_{\mathfrak Q^r}
 \asymp_{m,r}
 \left(
   \sum_{\substack{A\ \mathrm{word}\\\abs A\le r}}
   \norm{AG}_{L^2}^2
 \right)^{1/2}.
\tag{7.6}\label{eq:conjugation}
\]
\end{lemma}

\begin{proof}
First let $F\in C_c^\infty(\R^m)$ and set $G=\gamma_mF$.  For every
$\abs\alpha+\abs\gamma\le r$,
\[
 \norm{x^\alpha\partial^\gamma F}_{L^2(W_m)}
 =\pi^{-m/4}\norm{\gamma_mx^\alpha\partial^\gamma F}_2.
\]
The identities
\[
 \partial_j(\gamma_mF)=\gamma_m(\partial_j-x_j)F,
 \qquad
 \gamma_m\partial_jF=(\partial_j+x_j)(\gamma_mF).
\tag{7.7}\label{eq:conjugationid}
\]
iterated at most $r$ times express each function on the right of the first
display as a finite linear combination of $AG$, where $A$ is a word of
length at most $r$ in $X_j,\partial_j$.  The inverse identities express each
such $AG$ as a finite linear combination of
$\gamma_mx^\alpha\partial^\gamma F$ with
$\abs\alpha+\abs\gamma\le r$.  Hence
\[
 \norm F_{\mathsf S_r}
 \asymp_{m,r}
 \left(\sum_{\abs A\le r}\norm{AG}_2^2\right)^{1/2}.
\]
Repeated use of
$[\partial_j,X_k]=\one_{j=k}I$ expresses every word as a finite linear
combination of ordered monomials $x^\alpha\partial^\gamma$ of no greater
order; conversely, each ordered monomial is a word.  This proves the second
equivalence on $C_c^\infty$.  Since multiplication by $\gamma_m$ is a
bijection of $C_c^\infty$ and both inequalities hold in both directions, it
extends uniquely to an isomorphism between the corresponding graph-norm
completions, proving \eqref{eq:conjugation}.
\end{proof}

In one dimension, Lemma~\ref{lem:conjugation} and ordinary Sobolev embedding
give
\[
 \max_{A:\,\abs A\le3}\norm{AF}_{L^\infty}
 \le C\norm F_{\mathfrak Q^4},
\tag{7.8}\label{eq:wordembedding}
\]
where $A$ ranges over words in $X,\partial$.  Indeed, both $AF$ and
$\partial(AF)$ are finite sums of words of length at most four, and
$H^1(\R)\hookrightarrow L^\infty(\R)$.

\subsection{Linearized Hermite coercivity}

Gaussian conjugation has converted the weighted Sobolev norm into a graph
norm for multiplication and differentiation.  We now quantify the part of
the defect that is linear in the source
perturbations.  The next lemma shows that standardization prevents
cancellation in this linearization: a rotation with mixing strength
$\abs{sc}$ creates an $\mathsf S_3$ signal of size at least $\abs{sc}\eta$.
This is the leading term to be compared with the nonlinear remainder in the
next subsection.

\begin{lemma}[Linearized Hermite coercivity]
\label{lem:linearcoercivity}
Let $Z_1,Z_2$ be independent, centered, variance-one real random variables
with $\norm{Z_j}_{\psi_2}<\infty$, and set
\[
 f_j(t):=\E e^{itZ_j},\qquad
 h_j:=f_j-g,\qquad
 q_j:=\frac{h_j}{g}.
\]
Set
\[
 \eta^2:=\sum_{j=1}^2\norm{h_j}_{L^2(w)}^2.
\]
For $\theta\in\R$, put $c=\cos\theta$, $s=\sin\theta$, and define
\begin{align*}
 \mathcal L_{c,s}(q_1,q_2)(u,v)
 :={}&q_1(cu-sv)-q_1(cu)-q_1(-sv)\\
 &+q_2(su+cv)-q_2(su)-q_2(cv),\\
 D_{\mathrm{lin},\theta}(u,v)
 :={}&g(u)g(v)\mathcal L_{c,s}(q_1,q_2)(u,v).
\end{align*}
This is exactly the part of the rotated independence defect that is linear
in $(h_1,h_2)$.  There is a universal numerical constant $a_0>0$ such that
\[
 \norm{D_{\mathrm{lin},\theta}}_{\mathsf S_3(\R^2)}
 \ge a_0\abs{cs}\eta
 \qquad\text{for every }\theta\in\R.
\]
\end{lemma}

\emph{Proof idea.}
Expand each quotient $q_j$ in Hermite polynomials.  Under rotation, a
one-dimensional term of degree $m$ becomes a sum of bivariate Hermite terms
whose two coordinate degrees add to $m$.  Removing the terms that depend on
only one coordinate leaves a $2\times2$ Gram matrix for the two source
coefficients.  Its least eigenvalue is at least $3c^2s^2$ for $m\ge3$.
The mass, mean, and variance identities control the remaining degrees
$0,1,2$.

\begin{proof}
Fix $\theta$ and abbreviate $c=\cos\theta$, $s=\sin\theta$.  The full defect
is
\begin{equation}
 D_\theta(u,v)
 :=f_1(cu-sv)f_2(su+cv)
 -f_1(cu)f_2(su)f_1(-sv)f_2(cv).
\label{eq:JM}
\end{equation}
The terms linear in $(h_1,h_2)$ in the first product are
\[
 g(u)g(v)\{q_1(cu-sv)+q_2(su+cv)\},
\]
whereas those in the product of marginals are
\[
 g(u)g(v)\{q_1(cu)+q_2(su)+q_1(-sv)+q_2(cv)\}.
\]
Their difference is the stated $D_{\mathrm{lin},\theta}$.
Put
\[
 \dd\mu(t):=\sqrt{\frac{2}{\pi}}e^{-2t^2}\dd t,
 \qquad
 e_m(t):=\frac{\He_m(2t)}{\sqrt{m!}},\quad m\ge0.
\]
We use the probabilists' Hermite polynomials, characterized by
\[
 \sum_{m=0}^{\infty}\He_m(x)\frac{z^m}{m!}=e^{xz-z^2/2}.
\]
Then $(e_m)_{m\ge0}$ is an orthonormal basis of $L^2(\mu)$, and
\[
 \norm{h_j}_{L^2(w)}^2
 =2^{-1/2}\norm{q_j}_{L^2(\mu)}^2.
\]
Write $q_j=\sum_{m\ge0}\alpha_{j,m}e_m$ and, for
$q=\sum_{m\ge0}\alpha_me_m$, set
\[
 \norm q_{\cH_3}^2
 :=\sum_{m\ge0}(1+m)^3\abs{\alpha_m}^2.
\]
To compare the Hermite coefficients with the $\mathsf S_3$ norm, define
operators on smooth functions by
\[
 (\mathsf Xq)(t):=tq(t),
 \qquad
 (\mathsf Yq)(t):=q'(t)-tq(t).
\]
Let $\mathcal W_3(q)^2$ be the sum of the squared $L^2(\mu)$ norms of
$Aq$ over all words $A$ of length at most three in $\mathsf X,\mathsf Y$.
Since
\[
 g\,\mathsf Xq=t(gq),
 \qquad
 g\,\mathsf Yq=\partial_t(gq),
 \qquad
 [\mathsf Y,\mathsf X]=I,
\]
the ordered monomials defining $\mathsf S_3$ and these word operators span
one another with universal coefficients.  Hence
\[
 \norm{gq}_{\mathsf S_3(\R)}\asymp\mathcal W_3(q).
\]
The pair $(\mathsf X,\mathsf Y)$ is an invertible linear combination of
\[
 \mathfrak a_-:=\tfrac12\partial_t,
 \qquad
 \mathfrak a_+:=2t-\tfrac12\partial_t,
\]
and
\[
 \mathfrak a_-e_m=\sqrt m\,e_{m-1},\qquad
 \mathfrak a_+e_m=\sqrt{m+1}\,e_{m+1}.
\]
For every finite Hermite sum, these identities imply
\[
 \mathcal W_3(q)^2
 \le C\sum_{m\ge0}(1+m)^3\abs{\alpha_m}^2.
\]
Conversely,
\[
 \norm{\mathfrak a_-^3q}_{L^2(\mu)}^2
 =\sum_{m\ge3}m(m-1)(m-2)\abs{\alpha_m}^2,
\]
so the empty-word term gives
\[
 \sum_{m\ge0}(1+m)^3\abs{\alpha_m}^2
 \le C\{\norm q_{L^2(\mu)}^2
          +\norm{\mathfrak a_-^3q}_{L^2(\mu)}^2\}
 \le C\mathcal W_3(q)^2.
\]
Define the common maximal graph domain
\[
 \mathcal D_3
 :=\left\{q\in L^2(\mu):
 Aq\in L^2(\mu)\ \text{distributionally for every word }
 A\ \text{with }\abs A\le3\right\}.
\]
Let $q_N$ be the $N$th Hermite truncation of $q$.  If $q\in\cH_3$, the
upper estimate makes $(Aq_N)_N$ Cauchy in $L^2(\mu)$ for every word
$\abs A\le3$.  Since $q_N\to q$ in $L^2(\mu)$, the corresponding limits
equal $Aq$ distributionally; hence $q\in\mathcal D_3$.  Conversely, if
$q\in\mathcal D_3$, then $\mathfrak a_-^3q\in L^2(\mu)$ because
$\mathfrak a_-$ is a linear combination of $\mathsf X$ and $\mathsf Y$.
Writing $q=\sum_{m\ge0}\alpha_me_m$, the distributional adjoint identity
gives
\[
 \left\langle\mathfrak a_-^3q,e_k\right\rangle
 =\sqrt{(k+1)(k+2)(k+3)}\,\alpha_{k+3}.
\]
Parseval and the empty-word term therefore imply
$\sum_m(1+m)^3\abs{\alpha_m}^2<\infty$, so $q\in\cH_3$.
Thus $\mathcal D_3=\cH_3$ with equivalent graph norms.  Therefore
\[
 \norm{gq}_{\mathsf S_3(\R)}\asymp\norm q_{\cH_3}.
\]
Finally, $\abs{f_j^{(b)}(t)}\le\E\abs{Z_j}^b$ for $0\le b\le3$.
Gaussian integrability and the finite sub-Gaussian moments give
$h_j\in\mathsf S_3$ and hence $q_j\in\cH_3$.

For
\[
 Q=\sum_{k,\ell\ge0}\widehat Q_{k,\ell}e_k\otimes e_\ell,
 \qquad
 \norm Q_{\cH_3^{(2)}}^2
 :=\sum_{k,\ell\ge0}(1+k+\ell)^3
   \abs{\widehat Q_{k,\ell}}^2,
\]
the tensorized equivalence is
\[
 \norm{g(u)g(v)Q(u,v)}_{\mathsf S_3(\R^2)}
 \asymp\norm Q_{\cH_3^{(2)}}.
\]

Unit mass, centering, and unit variance give
\begin{equation}
 h_j(0)=h_j'(0)=h_j''(0)=0.
\label{eq:traces}
\end{equation}
Let $\Pi_{\rm mix}$ be the orthogonal projection onto the span of
$e_k\otimes e_\ell$ with $k,\ell\ge1$.  It removes all one-coordinate terms
from $\mathcal L_{c,s}$.  The Hermite addition formulas are
\begin{align*}
 e_m(cu-sv)
 &=\sum_{k=0}^m\binom{m}{k}^{1/2}
   c^k(-s)^{m-k}e_k(u)e_{m-k}(v),\\
 e_m(su+cv)
 &=\sum_{k=0}^m\binom{m}{k}^{1/2}
   s^kc^{m-k}e_k(u)e_{m-k}(v).
\end{align*}
Set
\[
 v_{1,m}:=\Pi_{\rm mix}e_m(cu-sv),\qquad
 v_{2,m}:=\Pi_{\rm mix}e_m(su+cv).
\]
Set $G_0:=0_{2\times2}$.  For $m\ge1$, their Gram matrix is
\[
 G_m=
 \begin{pmatrix}A_m&B_m\\B_m&A_m\end{pmatrix},\qquad
 A_m=1-c^{2m}-s^{2m},\qquad
 B_m=\begin{cases}
 0,&m\text{ odd},\\
 -2(sc)^m,&m\text{ even},
 \end{cases}
\]
because
\[
 \langle v_{1,m},v_{2,m}\rangle
 =(sc)^m\sum_{k=1}^{m-1}\binom{m}{k}(-1)^{m-k}.
\]
Thus the least eigenvalue is
\[
 \lambda_m=
 \begin{cases}
 1-c^{2m}-s^{2m},&m\text{ odd},\\[1mm]
 1-(\abs c^m+\abs s^m)^2,&m\text{ even}.
 \end{cases}
\]
The matrices $G_m$ are positive semidefinite, $\lambda_1=\lambda_2=0$,
and, for every $m\ge3$,
\[
 \lambda_m\ge3c^2s^2.
\]
Indeed, with $x=c^2$, $y=s^2$, and $x+y=1$, odd $m\ge3$ give
$1-x^m-y^m\ge1-x^3-y^3=3xy$.  For even $m=2k\ge4$,
\[
 1-(x^k+y^k)^2
 \ge1-(x^2+y^2)^2
 =4xy(1-xy)\ge3xy.
\]

It remains to control degrees zero, one, and two.  Decompose
\[
 q_j=q_j^{\le2}+q_j^{>2},\qquad
 q_j^{\le2}\in K_2:=\operatorname{span}\{e_0,e_1,e_2\},\qquad
 q_j^{>2}\perp K_2,
\]
and define
\[
 \mathcal Eq:=\bigl((gq)(0),(gq)'(0),(gq)''(0)\bigr).
\]
In the basis $(e_0,e_1,e_2)$, the restriction of $\mathcal E$ to $K_2$
has matrix
\[
 \begin{pmatrix}
 1&0&-1/\sqrt2\\
 0&2&0\\
 -1&0&9/\sqrt2
 \end{pmatrix},
\]
whose determinant is $8\sqrt2$.  The evaluation map $\mathcal E$ is
continuous on $\cH_3$
by the graph-norm equivalence and $H^3(-1,1)\hookrightarrow C^2([-1,1])$.
Since \eqref{eq:traces} implies $\mathcal Eq_j=0$,
\[
 \mathcal Eq_j^{\le2}=-\mathcal Eq_j^{>2}.
\]
Boundedness of $(\mathcal E|_{K_2})^{-1}$ and continuity of
$\mathcal E:\cH_3\to\mathbb C^3$ therefore give
\begin{equation}
 \norm{q_j^{\le2}}_{\cH_3}
 \le C\norm{\mathcal Eq_j^{\le2}}_{\mathbb C^3}
 =C\norm{\mathcal Eq_j^{>2}}_{\mathbb C^3}
 \le C\norm{q_j^{>2}}_{\cH_3}.
\label{eq:lowcontrol}
\end{equation}

The projection $\Pi_{\rm mix}$ is a contraction on $\cH_3^{(2)}$, and
different total Hermite degrees are orthogonal.  Hence, with
$\boldsymbol\alpha_m=(\alpha_{1,m},\alpha_{2,m})^\top$,
\begin{align*}
 \norm{D_{\mathrm{lin},\theta}}_{\mathsf S_3}^2
 &\gtrsim
 \sum_{m\ge0}(1+m)^3
 \boldsymbol\alpha_m^*G_m\boldsymbol\alpha_m\gtrsim c^2s^2
 \sum_{j=1}^2\norm{q_j^{>2}}_{\cH_3}^2
 \gtrsim c^2s^2
 \sum_{j=1}^2\norm{q_j}_{L^2(\mu)}^2
 \asymp c^2s^2\eta^2.
\end{align*}
Here the penultimate inequality follows from \eqref{eq:lowcontrol}.
Taking square roots proves the claim, uniformly in $\theta$.
\end{proof}

\subsection{Uniform product estimates and the nonlinear remainder}

Lemma~\ref{lem:linearcoercivity} controls the linear term, but the theorem
concerns the full defect, so we also need to bound the terms containing two or
more source perturbations while retaining their vanishing factor
$\abs{sc}$.  The next two lemmas control products evaluated at scaled or
rotated arguments; Proposition~\ref{prop:nonlinear} then applies those
bounds to the exact nonlinear remainder.  We call each one-dimensional
factor in such a product a \emph{slot}.

\begin{lemma}[Two scaled slots]
\label{lem:scaledslots}
Fix $c_*>0$.  For $\theta\in\R$, write
$c_\theta=\cos\theta$, $s_\theta=\sin\theta$.  If
$\abs{c_\theta}\ge c_*$ and $F,G\in\mathfrak Q^4(\R)$, then, for
$P_\theta(x)=F(c_\theta x)G(s_\theta x)$,
\[
 \norm{P_\theta}_{\mathfrak Q^3}
 +\norm{\partial_\theta P_\theta}_{\mathfrak Q^3}
 \le C(c_*)\norm F_{\mathfrak Q^4}\norm G_{\mathfrak Q^4}.
\tag{7.24}\label{eq:scaledslots}
\]
The same statement holds with $s_\theta x$ replaced by $-s_\theta x$ and with the slots
interchanged.
\end{lemma}

\begin{proof}
First take $F,G$ in the Schwartz class $\cS(\R)$; the general case follows by smooth
cutoff and mollification, because Schwartz functions are dense in
$\mathfrak Q^4$.  Abbreviate $c=c_\theta$, $s=s_\theta$ and put
$z=cx,y=sx$.  Expanding a word in $X,\partial_x$ by the identities
\[
 x=cz+sy,\qquad \partial_x=c\partial_z+s\partial_y
\tag{7.25}\label{eq:slotidentities}
\]
shows that a word of length at most three is a finite linear combination,
with coefficients bounded by a numerical constant, of terms
$(AF)(cx)(BG)(sx)$ with $\abs A+\abs B\le3$.  By
\eqref{eq:wordembedding} and $\abs c\ge c_*$,
\[
 \norm{(AF)(c\,\cdot)(BG)(s\,\cdot)}_2
 \le \norm{(AF)(c\,\cdot)}_2\norm{BG}_\infty
 \le C(c_*)\norm F_{\mathfrak Q^4}\norm G_{\mathfrak Q^4}.
\tag{7.26}\label{eq:slotbasic}
\]

Since $\dot c=-s$ and $\dot s=c$,
\[
 \partial_\theta P_\theta
 =-F'(z)[yG(y)]+[zF(z)]G'(y).
\tag{7.27}\label{eq:slotangular}
\]
Applying a word of length at most three to \eqref{eq:slotangular} yields
terms $(AF)(cx)(BG)(sx)$ with
$0\le\abs A,\abs B\le4$ and $\abs A+\abs B\le5$.  If $\abs B\le3$, then
\[
 \norm{(AF)(c\,\cdot)(BG)(s\,\cdot)}_2
 \le C(c_*)\norm F_{\mathfrak Q^4}\norm G_{\mathfrak Q^4},
\]
by the same $L^2$--$L^\infty$ argument as in \eqref{eq:slotbasic}.  Each
summand in \eqref{eq:slotangular} already contains one word letter in each
slot.  Hence, if $\abs B=4$, then $\abs A\ge1$; because
$\abs A+\abs B\le5$, necessarily $\abs A=1$.  All three additional
operators must then contribute their second-slot terms in
\eqref{eq:slotidentities}, so the coefficient contains $s^3$.  For
$s\ne0$,
\[
 \abs s^3\norm{(AF)(c\,\cdot)(BG)(s\,\cdot)}_2
 \le \abs s^3\norm{AF}_\infty\norm{(BG)(s\,\cdot)}_2
 \le C\abs s^{5/2}\norm F_{\mathfrak Q^4}\norm G_{\mathfrak Q^4}.
\tag{7.28}\label{eq:slotsmall}
\]
At $s=0$ the corresponding coefficients are exactly zero.  Summing the
finitely many terms proves the estimate for Schwartz functions.  Density in
$\mathfrak Q^4$ and the same bilinear bounds extend both $P_\theta$ and its
displayed derivative to general $F,G\in\mathfrak Q^4$.
\end{proof}

\begin{lemma}[Orthogonally rotated slots]
\label{lem:rotatedslots}
For $\theta\in\R$, write $c_\theta=\cos\theta$ and
$s_\theta=\sin\theta$.  If
\[
 K_\theta(u,v)=F(c_\theta u-s_\theta v)
 G(s_\theta u+c_\theta v),\qquad F,G\in\mathfrak Q^4(\R),
\]
then
\[
 \norm{K_\theta}_{\mathfrak Q^3(\R^2)}
 +\norm{\partial_\theta K_\theta}_{\mathfrak Q^3(\R^2)}
 \le C\norm F_{\mathfrak Q^4}\norm G_{\mathfrak Q^4}.
\tag{7.29}\label{eq:rotatedslots}
\]
\end{lemma}

\begin{proof}
Abbreviate $c=c_\theta$, $s=s_\theta$ and use the orthogonal coordinates
$x=cu-sv,y=su+cv$.  Under this change, each multiplication or derivative
operator in $(u,v)$ is a linear combination, with coefficients bounded by
one, of the corresponding operators in $(x,y)$.  Expanding the finitely many
words of length at most three and applying Fubini therefore gives
\[
 \norm{K_\theta}_{\mathfrak Q^3(\R^2)}
 \le C\norm F_{\mathfrak Q^3}\norm G_{\mathfrak Q^3}.
\]
Moreover,
\[
 \partial_\theta K_\theta=-F'(x)[yG(y)]+[xF(x)]G'(y).
\]
After at most three further word letters are applied, each tensor factor has
word length at most four.  Its $L^2(\R^2)$ norm factors into the product of
the one-dimensional norms, proving \eqref{eq:rotatedslots}.
\end{proof}

We shall also use, directly from Fubini's theorem and the word norm,
\[
 \norm{U(u)V(v)}_{\mathfrak Q^3(\R^2)}
 \le C\norm U_{\mathfrak Q^3(\R)}\norm V_{\mathfrak Q^3(\R)}.
\tag{7.30}\label{eq:tensorq}
\]

The estimates above also supply the strong angular regularity needed when we
integrate an angular derivative or take a limit at an unmixed angle.

\begin{lemma}[Joint angular regularity of product maps]
\label{lem:angularregularity}
Fix $c_*>0$, and let $I$ be a connected component of
$\{\theta:\abs{\cos\theta}\ge c_*\}$.  The maps
\begin{align*}
 (F,G,\theta)&\longmapsto F(\cos\theta\,\cdot)G(\sin\theta\,\cdot)
 &&\text{from }\mathfrak Q^4(\R)^2\times I
   \text{ to }\mathfrak Q^3(\R),\\
 (F,G,\theta)&\longmapsto
 F(\cos\theta\,u-\sin\theta\,v)
 G(\sin\theta\,u+\cos\theta\,v)
 &&\text{from }\mathfrak Q^4(\R)^2\times I
   \text{ to }\mathfrak Q^3(\R^2)
\end{align*}
are jointly continuous.  Their angular derivatives, given respectively by
\eqref{eq:slotangular} and the derivative displayed in the proof of
Lemma~\ref{lem:rotatedslots}, are jointly continuous into the same target
spaces.  Consequently, for fixed $F,G$, both angular paths are strongly
$C^1$.  The same conclusions hold after a sign change, an interchange of
the two slots, or a tensor product of two one-dimensional scaled-slot maps.
\end{lemma}

\begin{proof}
First let $F,G$ belong to the Schwartz class $\cS(\R)$.  For the scaled-slot map, every word of order at
most three applied to the map or its angular derivative is a finite sum of
terms of the form
\[
 x^r F^{(a)}(\cos\theta\,x)G^{(b)}(\sin\theta\,x),
\]
with fixed finite ranges of $r,a,b$.  Since
$\abs{\cos\theta}\ge c_*$ on $I$, the first factor has uniform Schwartz
decay in $x$, while every derivative of the second factor is bounded.
Dominated convergence in each word norm proves continuity in $\theta$ for
the map and its displayed derivative.  For the rotated-slot map, the same
conclusion follows after the orthogonal change of variables used in
Lemma~\ref{lem:rotatedslots}; the relevant Schwartz seminorms are uniform in
$\theta$.

Now take $F,G\in\mathfrak Q^4$ and choose Schwartz approximations
$F_n\to F$, $G_n\to G$ in $\mathfrak Q^4$.  Bilinearity and
Lemmas~\ref{lem:scaledslots}--\ref{lem:rotatedslots} give, uniformly over
$\theta\in I$, bounds of the form
\[
 C\bigl(\norm{F_n-F}_{\mathfrak Q^4}\norm{G_n}_{\mathfrak Q^4}
 +\norm F_{\mathfrak Q^4}\norm{G_n-G}_{\mathfrak Q^4}\bigr)
\]
for both the maps and their angular derivatives.  For these approximations,
write
\begin{align*}
 P_{\theta,n}(x)
 &:=F_n(\cos\theta\,x)G_n(\sin\theta\,x),\\
 K_{\theta,n}(u,v)
 &:=F_n(\cos\theta\,u-\sin\theta\,v)
   G_n(\sin\theta\,u+\cos\theta\,v).
\end{align*}
The maps and their derivatives converge uniformly in the target Banach
spaces.  Passing to the limit in the identities
\[
 P_{\theta_2,n}-P_{\theta_1,n}
 =\int_{\theta_1}^{\theta_2}\partial_tP_{t,n}\dd t
\]
and in the analogous identity for $K_{\theta,n}$ gives the same identities
for the limits; continuity of the limiting derivatives
then proves strong $C^1$ regularity.  The bilinear estimates also give joint
continuity in $(F,G)$, hence in $(F,G,\theta)$.  Sign changes and slot
interchange are isometries, and the tensor-product conclusion follows from
\eqref{eq:tensorq}.
\end{proof}

\begin{lemma}[Joint angular regularity of the defect]
\label{lem:defectregularity}
Fix $c_*>0$ and let $I$ be a connected component of
$\{\theta:\abs{\cos\theta}\ge c_*\}$.  For
$f_1,f_2\in\mathsf S_4(\R)$, define $D_\theta$ by \eqref{eq:JM}.  Then
\[
 (f_1,f_2,\theta)\longmapsto D_\theta,
 \qquad
 (f_1,f_2,\theta)\longmapsto\partial_\theta D_\theta
\]
are jointly continuous from $\mathsf S_4(\R)^2\times I$ into
$\mathsf S_3(\R^2)$.  In particular,
$\theta\mapsto D_\theta$ is strongly $C^1$.
\end{lemma}

\begin{proof}
Abbreviate $c=\cos\theta$, $s=\sin\theta$, and write
\begin{align*}
 J_\theta(u,v)&:=f_1(cu-sv)f_2(su+cv),\\
 M_\theta(u,v)&:=f_1(cu)f_2(su)f_1(-sv)f_2(cv),
\end{align*}
so that $D_\theta=J_\theta-M_\theta$.  Set $F_j:=gf_j$.
Gaussian conjugation and
$g(u)g(v)=g(cu-sv)g(su+cv)$ give
\[
 \gamma_2J_\theta
 =F_1(cu-sv)F_2(su+cv).
\]
Likewise, $g(u)=g(cu)g(su)$ and $g(v)=g(-sv)g(cv)$ give
\[
 \gamma_2M_\theta
 =\{F_1(cu)F_2(su)\}\{F_1(-sv)F_2(cv)\}.
\]
Lemma~\ref{lem:conjugation} maps $f_j$ continuously from
$\mathsf S_4$ to $F_j\in\mathfrak Q^4$.  The first display has the asserted
regularity by Lemma~\ref{lem:angularregularity}.  For the second display,
apply the scaled-slot part of that lemma to each brace and then use the
continuous tensor-product map \eqref{eq:tensorq}.  A final application of
Lemma~\ref{lem:conjugation} transfers both conclusions to
$\mathsf S_3(\R^2)$.
\end{proof}

\begin{proposition}[Angle-factored nonlinear remainder]
\label{prop:nonlinear}
Fix $c_*>0$ and $M<\infty$.  Suppose
\[
 h_1,h_2\in\mathsf S_{12}(\R),\qquad h_1(0)=h_2(0)=0,
 \qquad \max_j\norm{h_j}_{\mathsf S_{12}}\le M.
\tag{7.31}\label{eq:nonlinearhyp}
\]
Set $f_j:=g+h_j$, $c=\cos\theta$, $s=\sin\theta$, and define
\begin{align*}
 J_\theta(u,v)&:=f_1(cu-sv)f_2(su+cv),\\
 M_\theta(u,v)&:=f_1(cu)f_2(su)f_1(-sv)f_2(cv),\\
 D_\theta&:=J_\theta-M_\theta.
\end{align*}
Let $D_{\mathrm{lin},\theta}$ be the part linear in $(h_1,h_2)$, set
$R_\theta:=D_\theta-D_{\mathrm{lin},\theta}$, and put
\[
 \eta^2:=\norm{h_1}_{L^2(w)}^2+\norm{h_2}_{L^2(w)}^2.
\]
If $\abs c\ge c_*$, then
\[
 \norm{\partial_\theta R_\theta}_{\mathsf S_3}
 \le C(c_*,M)
 \left(\norm{h_1}_{\mathsf S_4}+\norm{h_2}_{\mathsf S_4}\right)^2.
\tag{7.32}\label{eq:Rderivative}
\]
Moreover,
\[
 \norm{R_\theta}_{\mathsf S_3}
 \le C(c_*,M)\abs{sc}\eta^{4/3}.
\tag{7.33}\label{eq:Rbound}
\]
\end{proposition}

\begin{proof}
Put
\[
 x_1=cu,\quad x_2=su,\quad x_3=-sv,\quad x_4=cv,
 \qquad j(1)=j(3)=1,\quad j(2)=j(4)=2.
\tag{7.34}\label{eq:xslots}
\]
Rotational invariance gives
$g(cu-sv)g(su+cv)=g(u)g(v)$ and
$\prod_{k=1}^4g(x_k)=g(u)g(v)$.  Substitute $f_j=g+h_j$ into
$J_\theta$ and $M_\theta$, expand both products, and subtract their constant
and linear terms.  This gives
\begin{align}
 R_\theta(u,v)
={}&h_1(cu-sv)h_2(su+cv)-\sum_{\substack{A\subset\{1,2,3,4\}\\\abs A\ge2}}
 \left\{\prod_{k\in A}h_{j(k)}(x_k)
 \prod_{k\notin A}g(x_k)\right\}.
\tag{7.35}\label{eq:Rexpansion}
\end{align}
Every summand contains at least two $h$ factors.

Let $\widetilde h_j=gh_j$ and $\Gamma=g^2=e^{-t^2}$.  Since
$\gamma_2(u,v)=g(u)g(v)$, conjugating the joint quadratic term gives
\[
 \gamma_2h_1(cu-sv)h_2(su+cv)
 =\widetilde h_1(cu-sv)\widetilde h_2(su+cv),
\tag{7.36}\label{eq:jointconjugated}
\]
whereas a marginal term indexed by $A$ becomes
\[
 \gamma_2
 \prod_{k\in A}h_{j(k)}(x_k)\prod_{k\notin A}g(x_k)
 =\prod_{k\in A}\widetilde h_{j(k)}(x_k)
  \prod_{k\notin A}\Gamma(x_k).
\tag{7.37}\label{eq:margconjugated}
\]
The right side of \eqref{eq:margconjugated} is a tensor product of a two-slot
function of $u$ and a two-slot function of $v$.  Lemma~\ref{lem:conjugation}
gives
\[
 \norm{\widetilde h_j}_{\mathfrak Q^4}
 \le C\norm{h_j}_{\mathsf S_4},
\tag{7.38}\label{eq:htilde}
\]
while $\Gamma$ has bounded Shubin norm of every fixed order.  Apply
Lemma~\ref{lem:rotatedslots} to \eqref{eq:jointconjugated}; apply
Lemma~\ref{lem:scaledslots} to the two coordinate factors in
\eqref{eq:margconjugated}, followed by \eqref{eq:tensorq}.  Define
\[
 U_J:=\gamma_2h_1(cu-sv)h_2(su+cv),
\qquad
 U_A:=\gamma_2
 \prod_{k\in A}h_{j(k)}(x_k)\prod_{k\notin A}g(x_k).
\]
For $U\in\{U_J\}\cup\{U_A:\abs A\ge2\}$, let $k(U)$ denote the number of
$h_j$ factors in $U$.  The preceding estimates give
\[
 \norm{\partial_\theta U}_{\mathfrak Q^3}
 \le C\left(\norm{h_1}_{\mathsf S_4}
              +\norm{h_2}_{\mathsf S_4}\right)^{k(U)}.
\tag{7.39}\label{eq:termdegree}
\]
Put $H_4=\norm{h_1}_{\mathsf S_4}+\norm{h_2}_{\mathsf S_4}$.  Since
$H_4\le2M$ and $k(U)\in\{2,3,4\}$,
\[
 H_4^{k(U)}\le\max\{1,(2M)^2\}H_4^2.
\]
Summing the joint term and the eleven marginal terms and conjugating back
proves \eqref{eq:Rderivative}.

Let $(\varphi_m)_{m\ge0}$ be the normalized Hermite-function basis of
$L^2(\R)$ and write
\[
 \widehat F_m:=\langle F,\varphi_m\rangle_{L^2}.
\]
The associated creation and annihilation identities, by the same calculation
as in Lemma~\ref{lem:linearcoercivity}, give
\[
 \norm F_{\mathfrak Q^r}^2\asymp
 \sum_{m\ge0}(1+m)^r\abs{\widehat F_m}^2.
\tag{7.40}\label{eq:ShubinHermite}
\]
If $a_m:=\langle gh_j,\varphi_m\rangle_{L^2}$, H\"older's
inequality with exponents $3/2$ and $3$ gives
\[
 \sum_m(1+m)^4\abs{a_m}^2
 \le
 \left(\sum_m\abs{a_m}^2\right)^{2/3}
 \left(\sum_m(1+m)^{12}\abs{a_m}^2\right)^{1/3}.
\]
Equation~\eqref{eq:ShubinHermite}, Gaussian conjugation, and taking square
roots therefore give
\[
 \norm{h_j}_{\mathsf S_4}
 \le C\norm{h_j}_{L^2(w)}^{2/3}
       \norm{h_j}_{\mathsf S_{12}}^{1/3}.
 \tag{7.41}\label{eq:interpolation}
\]
Writing $\eta_j=\norm{h_j}_{L^2(w)}$ and using
\eqref{eq:nonlinearhyp},
\[
 H_4
 \le CM^{1/3}(\eta_1^{2/3}+\eta_2^{2/3})
 \le C(M)(\eta_1^2+\eta_2^2)^{1/3}
 =C(M)\eta^{2/3},
 \qquad
 \norm{\partial_\theta R_\theta}_{\mathsf S_3}
 \le C(c_*,M)\eta^{4/3}.
\tag{7.42}\label{eq:Rderivativeeta}
\]

Let $\theta_0\in\pi\mathbb Z$ be the multiple of $\pi$ in the same connected
component of $\{t:\abs{\cos t}\ge c_*\}$ as $\theta$.  At $\theta_0$, the
two coordinates are the original independent coordinates up to sign, so
$D_{\theta_0}=D_{\mathrm{lin},\theta_0}=R_{\theta_0}=0$; here
$h_j(0)=0$ is used for the linear term.  The segment from $\theta_0$ to
$\theta$ stays in that component.  If $c_*=1$, then $\theta=\theta_0$.
If $0<c_*<1$, writing $x:=\abs{\theta-\theta_0}\le\arccos c_*<\pi/2$ gives
\[
 \abs{\sin\theta}=\sin x,
 \qquad
 \frac{x}{\sin x}\le
 \frac{\arccos c_*}{\sqrt{1-c_*^2}}.
\]
Thus, in either case,
\[
 \abs{\theta-\theta_0}\le C(c_*)\abs{\sin\theta}.
\tag{7.43}\label{eq:anglelength}
\]
By \eqref{eq:Rexpansion}, Lemma~\ref{lem:angularregularity}, and Gaussian
conjugation, $t\mapsto R_t$ is strongly $C^1$ from this angle component into
$\mathsf S_3$.  Hence
\[
 R_\theta=\int_{\theta_0}^{\theta}\partial_tR_t\dd t
 \quad\text{in }\mathsf S_3.
\]
Combining this identity with \eqref{eq:Rderivativeeta} and
\eqref{eq:anglelength} gives
\[
 \norm{R_\theta}_{\mathsf S_3}
 \le C\abs s\eta^{4/3}
 \le C(c_*)\abs{sc}\eta^{4/3},
\]
which is \eqref{eq:Rbound}.
\end{proof}

\subsection{Local conclusion and compact complement}

The linear coercivity and remainder estimates are useful only when
the sources are sufficiently close to Gaussian.  We now combine them to
define that local region, then treat its complement by compactness.  The
latter step uses Darmois--Skitovich at mixed angles and the Maxwell
equation at unmixed angles to exclude every possible zero.

\begin{proof}[Proof of Theorem~\ref{thm:maxwell}]
If $c_*>1$, the assertion is vacuous.  If $c_*=1$, then $\abs c=1$ forces
$s=0$, so both sides of \eqref{eq:maxwell} vanish.  Hence assume
$0<c_*<1$.  Let $f_j$ be
the characteristic function of $Z_j$, set $h_j=f_j-g$, let
$G\sim\cN(0,1)$, and put
\[
 \eta^2:=\NonG_w(Z_1)^2+\NonG_w(Z_2)^2.
\]
For integers $p,\ell\ge0$ with $p+\ell\le12$,
\[
 \abs{\partial^\ell h_j(t)}
 \le\E\abs{Z_j}^{\ell}+\E\abs G^{\ell}\le C_{K,\ell}.
\]
Multiplication by $\abs t^p$, squaring, and integration against $W_1$ give
\[
 \sup_j\norm{h_j}_{\mathsf S_{12}}\le M(K).
\tag{7.44}\label{eq:sourceS12}
\]
Lemma~\ref{lem:linearcoercivity} and Proposition~\ref{prop:nonlinear} imply
\[
 \norm{D_\theta}_{\mathsf S_3}
 \ge\abs{sc}\eta\{a_0-C\eta^{1/3}\}.
\tag{7.45}\label{eq:localD}
\]
Define
\[
 \eta_{\rm loc}
 :=\min\left\{1,
 \left(\frac{a_0}{2C(K,c_*)}\right)^3\right\},
\]
where $C(K,c_*)$ is the constant in \eqref{eq:localD}.  Then
\[
 \norm{D_\theta}_{\mathsf S_3}
 \ge\frac{a_0}{2}\abs{sc}\eta
 \qquad\text{whenever }\eta\le\eta_{\rm loc}.
\tag{7.46}\label{eq:localbound}
\]

We now prove a uniform bound on the complement.  Let
$\mathfrak F_K$ be the set of characteristic functions of centered,
variance-one variables satisfying $\norm Z_{\psi_2}\le K$.  We claim that this set is
compact in $\mathsf S_{r_0}(\R)$ for every fixed integer $r_0\ge0$.

To prove this, take an arbitrary sequence $(f_n)\subset\mathfrak F_K$ and
choose $Z_n$ with characteristic function $f_n$.  Fix $a>K$.  Since
\[
 C_{a,q}:=\sup_{x\ge0}x^q e^{-x^2/(2a^2)}<\infty,
\]
for $\abs z>R$ one has
\[
 \abs z^q
 \le C_{a,q}e^{-R^2/(2a^2)}e^{z^2/a^2}.
\]
Consequently, for every $q\ge0$,
\[
 \sup_n\E\!\left[
 \abs{Z_n}^q\one_{\{\abs{Z_n}>R\}}\right]
 \le2C_{a,q}e^{-R^2/(2a^2)}\longrightarrow0.
\]
Thus the laws of $Z_n$ are tight and the families $(\abs{Z_n}^q)_n$ are
uniformly integrable.  After passing to a subsequence, write
$Z_n\Rightarrow Z$ and let $f$ be the characteristic function of $Z$.
For $0\le q\le r_0$,
\[
 f_n^{(q)}(t)=\E[(iZ_n)^qe^{itZ_n}]
 \longrightarrow \E[(iZ)^qe^{itZ}]=f^{(q)}(t)
\]
pointwise.  Moreover,
$\abs{f_n^{(q)}(t)-f^{(q)}(t)}\le2C_{K,q}$.  Hence, for every
$p+q\le r_0$, dominated convergence gives
\[
 \int_\R\abs t^{2p}
 \abs{f_n^{(q)}(t)-f^{(q)}(t)}^2W_1(t)\dd t\longrightarrow0.
\]
Thus $f_n\to f$ strongly in $\mathsf S_{r_0}$.  Uniform integrability passes
the mean and variance to the limit.  Finally, for every $a>K$, the definition
of $\norm{Z_n}_{\psi_2}\le K$ gives
$\E e^{Z_n^2/a^2}\le2$.  Portmanteau gives
$\E e^{Z^2/a^2}\le2$, and monotone convergence as $a\downarrow K$ yields
$\E e^{Z^2/K^2}\le2$.  Hence the limit remains in $\mathfrak F_K$.

To include the unmixed angles in the compact parameter set, we extend the
normalized defect $D_\theta/(sc)$ continuously to $s=0$.
On each connected component of $\{\theta:\abs{\cos\theta}\ge c_*\}$, let
$\theta_0\in\pi\mathbb Z$ be its zero of $\sin\theta$.  For $sc\ne0$, set
$\overline E(f_1,f_2,\theta):=D_\theta/(sc)$; at $\theta=\theta_0$, set it
equal to $\partial_\theta D_{\theta_0}$.  These definitions agree
continuously.  Indeed, $D_{\theta_0}=0$ and the Banach-space fundamental
theorem of calculus gives
\[
 \frac{D_\theta}{sc}
 =\frac{\theta-\theta_0}{sc}
 \int_0^1
 \partial_\theta D_{\theta_0+r(\theta-\theta_0)}\dd r.
\tag{7.47}\label{eq:Eextension}
\]
Lemma~\ref{lem:defectregularity} makes the integrand in
\eqref{eq:Eextension} jointly continuous in the sources, the angle, and
$r$, with values in $\mathsf S_3(\R^2)$.  Since
\[
 \frac{\theta-\theta_0}{\sin\theta\cos\theta}\longrightarrow1
 \qquad(\theta\to\theta_0),
\]
the right side of \eqref{eq:Eextension} converges in $\mathsf S_3$ to
$\partial_\theta D_{\theta_0}$, jointly in the sources.  Thus
$\overline E$ is jointly continuous, including at the endpoint.

At $s=0,c=1$, the endpoint value obtained by differentiating
\eqref{eq:JM} is
\[
 \mathcal G(f_1,f_2)(u,v)
 :=u f_1(u)f_2'(v)-v f_1'(u)f_2(v).
\tag{7.48}\label{eq:Maxwellequation}
\]
At $s=0,c=-1$ it is $\mathcal G(f_1,f_2)(-u,-v)$; hence all endpoints differ
only by an isometric coordinate reflection.

This continuous extension has no zero with $\eta>0$.  At an interior angle, a zero
would give $D_\theta=0$, so the two nontrivial rotated forms are independent;
Lemma~\ref{lem:DS} forces both sources to be Gaussian.  At an endpoint,
suppose \eqref{eq:Maxwellequation} vanishes.  Since $f_j(0)=1$, both
characteristic functions are nonzero on some interval
$(-\delta,\delta)$.  Define
\[
 A_j(t):=\frac{f_j'(t)}{tf_j(t)},
 \qquad 0<\abs t<\delta.
\]
For nonzero $u,v$ in this interval, \eqref{eq:Maxwellequation} gives
\[
 A_1(u)=A_2(v).
\tag{7.49}\label{eq:Maxwellratio}
\]
Fixing first $v$ and then $u$ shows that both $A_j$ equal one constant.
Since $f_j(0)=1$, $f_j'(0)=0$, and $f_j''(0)=-1$,
\[
 \lim_{t\to0}A_j(t)=-1.
\]
Thus $f_j'(t)=-tf_j(t)$ on $(-\delta,\delta)$, and the initial condition
$f_j(0)=1$ gives $f_j(t)=e^{-t^2/2}$ there.  For every compact set of
$z\in\mathbb C$, sub-Gaussianity supplies an integrable bound for
$\abs{Z_j}^q e^{\abs{\operatorname{Im}z}\abs{Z_j}}$ for every $q\ge0$.
Therefore $z\mapsto\E e^{izZ_j}$ is entire, and the identity theorem gives
$f_j=g$ on $\R$.

The set
\[
\left\{(f_1,f_2,\theta)\in\mathfrak F_K^2\times[0,2\pi]:
 \abs{\cos\theta}\ge c_*,\ \eta\ge\eta_{\rm loc}\right\}
\]
is compact because $\eta$ is continuous in $\mathsf S_0$.  If it is empty,
there is nothing to prove on the complementary region.  Otherwise, the preceding Darmois--Skitovich and Maxwell arguments show that
$\overline E(f_1,f_2,\theta)\ne0$ throughout this compact set.  Since
$\eta\ge\eta_{\rm loc}>0$ there, the continuous function
$\frac{\norm{\overline E(f_1,f_2,\theta)}_{\mathsf S_3}}{\eta}$
is strictly positive and therefore attains a positive minimum. Consequently,
\[
 \norm{D_\theta}_{\mathsf S_3}\ge a_1\abs{sc}\eta
 \qquad(\eta\ge\eta_{\rm loc}).
\tag{7.50}\label{eq:globalbound}
\]
Setting $a_*:=\min\{a_0/2,a_1\}$ and combining
\eqref{eq:localbound}, \eqref{eq:globalbound}, and
$\eta^2=\NonG_w(Z_1)^2+\NonG_w(Z_2)^2$ gives \eqref{eq:maxwell}.
\end{proof}

\begin{remark}[Why third order appears]
At zero Sobolev order, the evaluation constraint $h''(0)=0$ is not
continuous, so unit-variance normalization cannot uniformly control the
degree-two Hermite coefficient.  In $\mathsf S_3$, the three evaluations
$h(0)$, $h'(0)$, and $h''(0)$ are continuous.  Equation
\eqref{eq:lowcontrol} then bounds the coefficients of degrees zero, one, and
two by those of degrees at least three, on which rotation is uniformly
coercive.
\end{remark}

\section{Population score gap in the wrong direction}
\label{sec:popgap}

Theorem~\ref{thm:maxwell} is stated for an abstract rotation of independent
standardized sources.  We now verify that the standardized residuals from
the wrong OLS direction have exactly this form, with mixing strength bounded
below by a constant multiple of $\beta$.  This converts Theorem~\ref{thm:maxwell} into the population score gap used by the test.

For any centered law $P$ with positive-definite covariance, define its two
standardized OLS pairs by
\begin{align}
 (S_{\dirf},T_{\dirf})
 &=\left(
 \frac{X}{\sqrt{\Var X}},
 \frac{Y-\alpha_{\dirf}X}
 {\sqrt{\Var(Y-\alpha_{\dirf}X)}}
 \right),
 &\alpha_{\dirf}&=\frac{\Cov(X,Y)}{\Var X},\notag\\
 (S_{\dirr},T_{\dirr})
 &=\left(
 \frac{Y}{\sqrt{\Var Y}},
 \frac{X-\alpha_{\dirr}Y}
 {\sqrt{\Var(X-\alpha_{\dirr}Y)}}
 \right),
 &\alpha_{\dirr}&=\frac{\Cov(X,Y)}{\Var Y}.
\tag{8.1}\label{eq:OLSpairs}
\end{align}
Let $J_d(P)$ be the $\mathsf S_3$ norm of the defect \eqref{eq:defect} for
the pair indexed by $d\in\{\dirf,\dirr\}$.

\begin{proposition}[Correct score zero, wrong score separated]
\label{prop:population-gap}
There is $\kappa>0$, depending only on the fixed constants, such that, for
$0<\beta\le\beta_0$,
\[
 \begin{array}{lll}
 P\in\cP_{\dirf}:&J_{\dirf}(P)=0,&
 J_{\dirr}(P)\ge\kappa\beta\nu,\\[1mm]
 P\in\cP_{\dirr}:&J_{\dirr}(P)=0,&
 J_{\dirf}(P)\ge\kappa\beta\nu.
 \end{array}
\tag{8.2}\label{eq:scoregap}
\]
\end{proposition}

\begin{proof}
Consider a forward law.  Its forward pair is exactly $(Z_1,Z_2)$, so its
defect vanishes.  Let
\[
 \lambda:=\frac{a\sigma_1}{\sigma_2}.
\]
Since
\[
 \Var(Y)=a^2\sigma_1^2+\sigma_2^2,
 \qquad
 \alpha_{\dirr}
 =\frac{a\sigma_1^2}{a^2\sigma_1^2+\sigma_2^2},
\]
one has
\[
 \Var(X-\alpha_{\dirr}Y)
 =\frac{\sigma_1^2\sigma_2^2}
        {a^2\sigma_1^2+\sigma_2^2}.
\]
Substitution into \eqref{eq:OLSpairs} therefore gives
\[
 (S_{\dirr},T_{\dirr})
 =\frac1{\sqrt{1+\lambda^2}}
 \bigl(\lambda Z_1+Z_2,\ Z_1-\lambda Z_2\bigr).
\tag{8.3}\label{eq:wrongrotation}
\]
After ordering the latent source vector as $(Z_2,Z_1)$, the two rows in
\eqref{eq:wrongrotation} are exactly $(c,s)$ and $(-s,c)$, where
\[
 c=(1+\lambda^2)^{-1/2},\qquad
 s=\lambda(1+\lambda^2)^{-1/2}.
\]
The fixed coefficient and scale bounds give
\[
 \abs\lambda\le\Lambda:=\frac{\overline a\,\overline\sigma}
 {\underline\sigma},\qquad
 \abs\lambda\ge\frac{\underline\sigma}{\overline\sigma}\beta.
\]
Consequently
\[
 \abs c\ge(1+\Lambda^2)^{-1/2}=:c_*>0,
 \qquad
 \abs{sc}=\frac{\abs\lambda}{1+\lambda^2}\ge c_1\beta.
\tag{8.4}\label{eq:scbeta}
\]
Theorem~\ref{thm:maxwell}, \eqref{eq:scbeta}, and
$\NonG_w(Z_j)\ge\nu$ give
\[
 J_{\dirr}(P)
 \ge a_*\abs{sc}
 \{\NonG_w(Z_1)^2+\NonG_w(Z_2)^2\}^{1/2}
 \ge a_*c_1\sqrt2\,\beta\nu.
\]
Thus the forward case holds with $\kappa=a_*c_1\sqrt2$.  Swapping $X$ and
$Y$ proves the reverse case with the same constant.
\end{proof}

\section{A uniform estimator of the Sobolev score}
\label{sec:estimator}

Proposition~\ref{prop:population-gap} reduces direction selection to
estimating the forward and reverse population scores.  Each score is obtained
by applying the OLS standardization determined by $\Sigma(P)$ and then taking
the norm of the mean of a Hilbert-valued characteristic-function feature.
We estimate $\Sigma(P)$ on one subsample and the two feature means on another,
so that the random standardization is independent of the observations used
for feature estimation.  The following Hilbert-space median-of-means lemma
provides the required conditional concentration.

\begin{proposition}[Measurable median-of-means selector in a Hilbert space]
\label{prop:robustmean}
Let $V_1,\ldots,V_N$ be i.i.d. random elements of a real or complex
separable Hilbert space $\mathbb H$, with mean $m$ and
$\E\norm{V_1-m}_{\mathbb H}^2\le v^2$.  For every integer
$8\le B\le N/2$, there is a Borel-measurable estimator
$\widehat m_B$ such that
\[
 \Pp\left(\norm{\widehat m_B-m}_{\mathbb H}
 >C v\sqrt{B/N}\right)\le2e^{-cB}.
\tag{9.1}\label{eq:HilbertMOM}
\]
The construction uses only the data and $B$, not $v$.
\end{proposition}

\begin{proof}
Partition the first $\ell B$ observations into $B$ blocks of common size
$\ell=\lfloor N/B\rfloor$, and let $M_b$ be the corresponding block means.
For each $b$, let $r_b$ be the $\lceil2B/3\rceil$-th order statistic of
\[
 \bigl\{\norm{M_b-M_j}_{\mathbb H}:1\le j\le B\bigr\}.
\]
Let $\widehat b$ be the smallest-index minimizer of $r_b$ and set
$\widehat m_B:=M_{\widehat b}$.  This estimator is Borel measurable because
it is obtained from finitely many measurable distances, order statistics,
and a deterministic tie-breaking rule.

Independence and the Hilbert-space inner product give
\[
 \E\norm{M_b-m}_{\mathbb H}^2
 =
 \frac1{\ell^2}
 \sum_{i=1}^{\ell}\E\norm{V_i-m}_{\mathbb H}^2
 \le\frac{v^2}{\ell}.
\]
Set
\[
 r:=\frac{\sqrt8\,v}{\sqrt\ell},
 \qquad
 I_b:=\one_{\{\norm{M_b-m}_{\mathbb H}>r\}}.
\]
Chebyshev's inequality gives $\E I_b\le1/8$.  The variables
$I_1,\ldots,I_B$ are independent, so Hoeffding's inequality yields
\[
 \Pp\left(\sum_{b=1}^B I_b>\frac B4\right)
 \le
 \Pp\left(
 \sum_{b=1}^B(I_b-\E I_b)>\frac B8
 \right)
 \le
 \exp\left\{-\frac{2(B/8)^2}{B}\right\}
 =e^{-B/32}.
\]
Thus, except on an event of probability at most $e^{-B/32}$, the set
\[
 \mathcal G
 :=\left\{b:\norm{M_b-m}_{\mathbb H}\le r\right\}
\]
contains at least $\lceil3B/4\rceil$ indices.

On this event, if $b,j\in\mathcal G$, then
\[
 \norm{M_b-M_j}_{\mathbb H}\le2r.
\]
Since $\lceil3B/4\rceil\ge\lceil2B/3\rceil$, every $b\in\mathcal G$
satisfies $r_b\le2r$, and hence
\[
 r_{\widehat b}\le2r.
\]
By the definition of $r_{\widehat b}$, at least
$\lceil2B/3\rceil$ indices $j$ satisfy
\[
 \norm{M_{\widehat b}-M_j}_{\mathbb H}\le r_{\widehat b}\le2r.
\]
This set of indices intersects $\mathcal G$, because
\[
 \left\lceil\frac{2B}{3}\right\rceil
 +\left\lceil\frac{3B}{4}\right\rceil>B.
\]
Choosing an index $j$ in the intersection gives
\[
 \norm{\widehat m_B-m}_{\mathbb H}
 \le
 \norm{M_{\widehat b}-M_j}_{\mathbb H}
 +\norm{M_j-m}_{\mathbb H}
 \le3r.
\]
Finally, $B\le N/2$ implies
\[
 \ell=\left\lfloor\frac NB\right\rfloor\ge\frac{N}{2B},
\]
and therefore
\[
 3r\le12v\sqrt{\frac BN}.
\]
Hence
\[
 \Pp\left(
 \norm{\widehat m_B-m}_{\mathbb H}
 >12v\sqrt{\frac BN}
 \right)
 \le e^{-B/32},
\]
which proves \eqref{eq:HilbertMOM}.
\end{proof}

To apply Proposition~\ref{prop:robustmean}, we express each directional score
as a function of a Hilbert-valued mean.  The relevant feature is evaluated at
the OLS-standardized pair $A_d(\Sigma(P))(X,Y)^\top$, so it also depends on
the unknown covariance matrix.  We begin by controlling this covariance
dependence uniformly over the model.

\paragraph{Uniform covariance conditioning.} Write
\[
 \Sigma(P):=\E_P[(X,Y)^\top(X,Y)]
\]
for the covariance matrix of the centered observation.  All model covariance
matrices lie in a fixed compact spectral interval: there are
$0<\lambda_-<\lambda_+<\infty$, depending only on the fixed constants, such that
\[
 \lambda_-I_2\preceq\Sigma(P)\preceq \lambda_+I_2
\tag{9.2}\label{eq:spectralcompact}
\]
for every law in either class.  Indeed, for a forward law,
\[
 (X,Y)^\top=B(Z_1,Z_2)^\top,
 \qquad
 B=\begin{pmatrix}\sigma_1&0\\a\sigma_1&\sigma_2\end{pmatrix},
 \qquad \Sigma=BB^\top.
\]
The fixed coefficient and scale bounds imply
$\norm{B}_{\rm op}+\norm{B^{-1}}_{\rm op}\le C$, and hence
\[
 \norm{B^{-1}}_{\rm op}^{-2}I_2
 \preceq \Sigma\preceq \norm{B}_{\rm op}^2I_2.
\]
For a reverse law the same argument applies to
$B=\left(\begin{smallmatrix}\tau_1&b\tau_2\\0&\tau_2\end{smallmatrix}\right)$
with source vector $(U_1,U_2)^\top$.  This proves
\eqref{eq:spectralcompact}.

\paragraph{OLS transformations.} For a positive-definite $2\times2$ matrix $\Gamma$, let
$A_d(\Gamma)$ be the matrix implementing the standardized OLS transformation
in direction $d$ from \eqref{eq:OLSpairs}.  Explicitly,
\begin{align}
 A_{\dirf}(\Gamma)
 &=\begin{pmatrix}
 \Gamma_{11}^{-1/2}&0\\
 -\dfrac{\Gamma_{12}/\Gamma_{11}}
 {\sqrt{\Gamma_{22}-\Gamma_{12}^2/\Gamma_{11}}}&
 \dfrac1{\sqrt{\Gamma_{22}-\Gamma_{12}^2/\Gamma_{11}}}
 \end{pmatrix},\notag\\[2mm]
 A_{\dirr}(\Gamma)
 &=\begin{pmatrix}
 0&\Gamma_{22}^{-1/2}\\
 \dfrac1{\sqrt{\Gamma_{11}-\Gamma_{12}^2/\Gamma_{22}}}&
 -\dfrac{\Gamma_{12}/\Gamma_{22}}
 {\sqrt{\Gamma_{11}-\Gamma_{12}^2/\Gamma_{22}}}
 \end{pmatrix}.
\tag{9.3}\label{eq:Atransform}
\end{align}
On
\[
 \mathcal C:=\{\Gamma:(\lambda_-/2)I_2\preceq\Gamma
 \preceq2\lambda_+I_2\},
\tag{9.4}\label{eq:Ccompact}
\]
these maps and their first derivatives are uniformly bounded.

Once an OLS transformation is fixed, all weighted derivatives entering the
$\mathsf S_3$ score can be represented as coordinates of a single
Hilbert-valued expectation.  We now construct that representation.

Write $\N_0:=\{0,1,2,\ldots\}$.  To describe the score estimator, set
(from this point through the end of
Section~\ref{sec:estimator}, $\alpha$ denotes a multiindex, never an OLS or
structural coefficient)
\begin{align*}
 \mathcal I_2&:=\{(\alpha,\gamma)\in\N_0^2\times\N_0^2:
                  \abs\alpha+\abs\gamma\le3\},\\
 \mathcal I_1&:=\{(a,k)\in\N_0\times\N_0:a+k\le3\}.
\end{align*}
Define the complex separable Hilbert spaces
\[
 \mathbb H_2:=\bigoplus_{(\alpha,\gamma)\in\mathcal I_2}L^2(W_2),
 \qquad
 \mathbb H_1:=\bigoplus_{(a,k)\in\mathcal I_1}L^2(W_1),
 \qquad
 \mathbb H:=\mathbb H_2\oplus\mathbb H_1\oplus\mathbb H_1.
\tag{9.5}\label{eq:featureHilbert}
\]

\paragraph{Hilbert-valued characteristic-function features.}
Fix an OLS-transformed pair $(S,T)$.  Each weighted derivative appearing in
the $\mathsf S_3$ score is the expectation of an explicit function of
$(S,T)$.  We collect all these functions into a single Hilbert-valued
feature.
For a pair $(S,T)$, define the finite direct-sum Hilbert feature
\[
 \Psi(S,T):=
 \left(
 \bigl[u^{\alpha_1}v^{\alpha_2}(iS)^{\gamma_1}(iT)^{\gamma_2}
 e^{i(uS+vT)}\bigr]_{(\alpha,\gamma)\in\mathcal I_2},
 \bigl[u^a(iS)^ke^{iuS}\bigr]_{(a,k)\in\mathcal I_1},
 \bigl[v^a(iT)^ke^{ivT}\bigr]_{(a,k)\in\mathcal I_1}
 \right),
\tag{9.6}\label{eq:PsiFeature}
\]
which is an element of $\mathbb H$.  For each coordinate of $\Psi(S,T)$, its
squared Hilbert norm is a finite Gaussian integral of a frequency monomial
multiplied by a polynomial in $(S,T)$ of total degree at most six.  Summing
over the finite index sets therefore gives
\[
 \norm{\Psi(S,T)}_{\mathbb H}^2
 \le C(1+\abs S^6+\abs T^6).
\tag{9.7}\label{eq:featuremoment}
\]

\paragraph{Recovering the Sobolev score.}
The mean of $\Psi(S,T)$ contains the joint and marginal characteristic-function
derivatives as separate coordinates.  The following map subtracts the
products of the marginal coordinates from the corresponding joint
coordinates.
For $z=(z^{(2)},z^{(S)},z^{(T)})\in\mathbb H$, define the locally Lipschitz
map $\mathcal T:\mathbb H\to\mathbb H_2$ componentwise by
\[
 [\mathcal Tz]_{\alpha,\gamma}(u,v)
 :=z^{(2)}_{\alpha,\gamma}(u,v)
 -z^{(S)}_{\alpha_1,\gamma_1}(u)
  z^{(T)}_{\alpha_2,\gamma_2}(v).
\tag{9.8}\label{eq:Tmap}
\]
The product lies in $L^2(W_2)$ by Fubini.  For later reference, on every
radius-$R$ ball in $\mathbb H$, tensor
factorization gives
\[
 \norm{\mathcal Tz-\mathcal Tz'}_{\mathbb H_2}
 \le C(1+R)\norm{z-z'}_{\mathbb H}
 \qquad(\norm z_{\mathbb H},\norm{z'}_{\mathbb H}\le R).
\tag{9.9}\label{eq:Tlipschitz}
\]
If $m_{S,T}:=\E\Psi(S,T)$, then
\[
 [\mathcal Tm_{S,T}]_{\alpha,\gamma}
 =u^{\alpha_1}v^{\alpha_2}
   \partial_u^{\gamma_1}\partial_v^{\gamma_2}D(u,v),
 \qquad
 J(S,T):=\norm{\mathcal Tm_{S,T}}_{\mathbb H_2}.
\tag{9.10}\label{eq:scoreasmap}
\]
Indeed, differentiation under the expectation gives the joint derivative,
and the last two feature groups give the two marginal derivatives.  Thus
\eqref{eq:scoreasmap} is exactly, component by component, the norm
\eqref{eq:Sr} of the defect \eqref{eq:defect}.
\\

This representation applies to a fixed OLS transformation.  In the estimator,
$\Sigma(P)$ is replaced by an estimate, so we must control how the feature
mean and the resulting score vary with the covariance matrix.

\begin{lemma}[Stability with respect to the OLS transform]
\label{lem:transformstability}
Uniformly over the model, over $d\in\{\dirf,\dirr\}$, and over
$\Gamma,\Gamma'\in\mathcal C$,
\[
 \norm{\E\Psi(A_d(\Gamma)(X,Y)^\top)
          -\E\Psi(A_d(\Gamma')(X,Y)^\top)}_{\mathbb H}
 \le C\norm{\Gamma-\Gamma'}_F.
\tag{9.11}\label{eq:transformstability}
\]
If
\[
 m_d(\Gamma):=\E\Psi(A_d(\Gamma)(X,Y)^\top),
 \qquad
 J_{d,\Gamma}(P):=\norm{\mathcal Tm_d(\Gamma)}_{\mathbb H_2},
\]
then
\[
 \abs{J_{d,\Gamma}(P)-J_{d,\Gamma'}(P)}
 \le C\norm{\Gamma-\Gamma'}_F.
\]
\end{lemma}

\begin{proof}
Each formula in \eqref{eq:Atransform} is $C^1$ on the positive-definite cone.
The compact convex set $\mathcal C$ lies inside that cone, so its derivatives
are bounded on a neighborhood of $\mathcal C$.  The finite-dimensional
mean-value theorem gives
\[
 \norm{A_d(\Gamma)-A_d(\Gamma')}_{\rm op}
 \le C\norm{\Gamma-\Gamma'}_F.
\]
Let $W=(X,Y)^\top$, and set
\[
 A_t=(1-t)A_d(\Gamma)+tA_d(\Gamma'),\qquad
 (S_t,T_t)^\top=A_tW,\qquad
 \dot A=A_d(\Gamma')-A_d(\Gamma).
\]
For example, differentiation of a joint component in
\eqref{eq:PsiFeature} gives
\begin{align*}
 &\quad \partial_t\left[u^{\alpha_1}v^{\alpha_2}(iS_t)^{\gamma_1}
 (iT_t)^{\gamma_2}e^{i(uS_t+vT_t)}\right]\\
 &=u^{\alpha_1}v^{\alpha_2}e^{i(uS_t+vT_t)}
 \Bigl[\gamma_1i\dot S(iS_t)^{\gamma_1-1}(iT_t)^{\gamma_2}
 +\gamma_2i\dot T(iS_t)^{\gamma_1}(iT_t)^{\gamma_2-1}
 +i(u\dot S+v\dot T)(iS_t)^{\gamma_1}(iT_t)^{\gamma_2}\Bigr],
\end{align*}
where the term multiplied by $\gamma_j$ is omitted when $\gamma_j=0$, and
$(\dot S,\dot T)^\top=\dot A W$.  The marginal components have the analogous
formula.  Since $\sup_t\norm{A_t}_{\rm op}\le C$ and
$\abs{\dot S}+\abs{\dot T}\le C\norm{\dot A}_{\rm op}\norm{W}_{\ell_2}$, Gaussian
integration over the frequency variables yields
\[
 \norm{\partial_t\Psi(S_t,T_t)}_{\mathbb H}
 \le C\norm{\dot A}_{\rm op}(1+\norm{W}_{\ell_2}^4).
\]
For each fixed $W$, dominated convergence in the Gaussian frequency
variables, using this polynomial envelope, shows that
$t\mapsto\Psi(S_t,T_t)$ is strongly $C^1$ as an $\mathbb H$-valued map, with
the derivative displayed above.
The pointwise Hilbert-valued fundamental theorem of calculus, Minkowski's
inequality, and the uniform eighth moment therefore give
\begin{align*}
 \norm{m_d(\Gamma')-m_d(\Gamma)}_{\mathbb H}
 &\le \int_0^1\E\norm{\partial_t\Psi(S_t,T_t)}_{\mathbb H}\dd t\\
 &\le C\norm{\Gamma-\Gamma'}_F,
\end{align*}
which is \eqref{eq:transformstability}. Moreover, Jensen's inequality and \eqref{eq:featuremoment} imply
\[
 \sup_{P,d,\Gamma\in\mathcal C}\norm{m_d(\Gamma)}_{\mathbb H}
 \le \sup_{P,d,\Gamma\in\mathcal C}
 \bigl(\E\norm{\Psi(A_d(\Gamma)W)}_{\mathbb H}^2\bigr)^{1/2}
 =:R_0<\infty.
\]
Applying \eqref{eq:Tlipschitz} on the radius-$R_0$ ball and then the reverse
triangle inequality proves the asserted bound for $J_{d,\Gamma}$.
\end{proof}

We can now combine covariance estimation, conditional robust mean estimation,
and the stability lemma.  The result is a uniform estimator of both
directional scores at the parametric deviation scale $\sqrt{x/n}$.

\begin{proposition}[Uniform high-probability score estimation]
\label{prop:score-estimator}
There are constants $c_0,c,C>0$ such that, for every
$8\le x\le c_0n$, there are measurable estimators
$\widehat J_{\dirf,x},\widehat J_{\dirr,x}$ satisfying
\[
 \sup_{P\in\cP_{\dirf}\cup\cP_{\dirr}}
 \Pp_P\left(
 \max_{d\in\{\dirf,\dirr\}}
 \abs{\widehat J_{d,x}-J_d(P)}>C\sqrt{x/n}
 \right)
 \le Ce^{-cx}.
\tag{9.12}\label{eq:scoreconcentration}
\]
\end{proposition}

\begin{proof}
Split the sample into parts of sizes $n_1=\lfloor n/3\rfloor$ and
$n_2=n-n_1$, write $W=(X,Y)^\top$, and set $B_x=\lceil x\rceil$.
Choose $c_0$ small enough that
$B_x\le\min(n_1,n_2)/2$ whenever $8\le x\le c_0n$.  On the first part, apply
Proposition~\ref{prop:robustmean}, with $B=B_x$, to the
$\R^3$-valued vector $(X^2,XY,Y^2)$.  The model assumes
$\E X=\E Y=0$, so its mean is exactly the three distinct entries of
$\Sigma$; its covariance is uniformly bounded by
the fourth-moment bound.  Symmetrize the result and project it in Frobenius
norm onto $\mathcal C$.  Projection onto a closed convex set is measurable and
nonexpansive relative to the true covariance, so, with probability at least
$1-Ce^{-cx}$,
\[
 \norm{\widehat\Sigma-\Sigma}_F\le C\sqrt{x/n}.
\tag{9.13}\label{eq:SigmaConcentration}
\]
Here, on the stated nonempty range (after decreasing $c_0$ if necessary),
\[
 B_x\le x+1\le\frac98x,\qquad n_1\ge\frac n4,\qquad n_2\ge\frac n2,
\]
so Proposition~\ref{prop:robustmean} indeed gives the displayed
$\sqrt{x/n}$ radius and an $e^{-cx}$ tail.

Conditional on this first subsample, transform each observation in the
remaining subsample by $A_d(\widehat\Sigma)$, for both $d$.  Apply
Proposition~\ref{prop:robustmean} in the direct-sum Hilbert space
$\mathbb H$, again with $B=B_x$, and denote its output in direction $d$ by
$\widehat m_{d,x}$.  Independence of the two subsamples implies
that, conditionally on $\widehat\Sigma$, the remaining observations are
still i.i.d. with law $P$.  Because $\widehat\Sigma\in\mathcal C$,
$\sup_{d,\Gamma\in\mathcal C}\norm{A_d(\Gamma)}_{\rm op}<\infty$; hence
\eqref{eq:featuremoment} and the uniform sixth moment give
\[
 \E\!\left[
 \norm{\Psi(A_d(\widehat\Sigma)W)-m_d(\widehat\Sigma)}_{\mathbb H}^2
 \mid\widehat\Sigma\right]
 \le C
\]
uniformly in $P,d$, and $\widehat\Sigma$.  A union bound over the two
directions therefore gives, except on an event of conditional probability
$Ce^{-cx}$,
\[
 \max_d\norm{\widehat m_{d,x}-m_d(\widehat\Sigma)}_{\mathbb H}
 \le C\sqrt{x/n}.
\]

Define
\[
 \widehat J_{d,x}:=\norm{\mathcal T\widehat m_{d,x}}_{\mathbb H_2}.
\]
The proof of Lemma~\ref{lem:transformstability} gives
$\norm{m_d(\widehat\Sigma)}_{\mathbb H}\le R_0$.  On the median-of-means
event,
\[
 \norm{\widehat m_{d,x}}_{\mathbb H}
 \le R_0+C\sqrt{x/n}\le R_0+C\sqrt{c_0}=:R.
\]
Thus both arguments lie in a fixed radius-$R$ ball, and
\eqref{eq:Tlipschitz}, \eqref{eq:transformstability}, and
\eqref{eq:SigmaConcentration} give
\begin{align*}
 \abs{\widehat J_{d,x}-J_d(P)}
 &\le
 \abs{\widehat J_{d,x}-J_{d,\widehat\Sigma}(P)}
 +\abs{J_{d,\widehat\Sigma}(P)-J_{d,\Sigma}(P)}\\
 &\le C\norm{\widehat m_{d,x}-m_d(\widehat\Sigma)}_{\mathbb H}
      +C\norm{\widehat\Sigma-\Sigma}_F
 \le C\sqrt{x/n}.
\end{align*}
Taking expectations of the conditional failure bound and adding the
covariance-estimation failure probability proves
\eqref{eq:scoreconcentration}.  Measurability follows from the finite
tie-broken selector in Proposition~\ref{prop:robustmean} and continuous
algebraic operations.
\end{proof}

The population gap from Section~\ref{sec:popgap} and the uniform score
estimator now give the non-Gaussian branch of the test.  We combine it with
the covariance branch from Proposition~\ref{prop:covariance} according to
which signal is larger.

\begin{proposition}[Uniform directional upper bound]
\label{prop:upperrisk}
For all sufficiently small $\beta,\nu$ and every $n\ge1$, there is a
measurable rule satisfying
\[
 \mathcal R_n(\widehat d_n;\beta,\nu)
 \le C\exp\{-cn(d_\beta^2+\beta^2\nu^2)\}.
\tag{9.14}\label{eq:upperrisk}
\]
\end{proposition}

The parameters $(\beta,\nu)$ and the fixed nuisance bounds specify the
minimax class and are available to the decision rule.  The rule may use them
to choose between the covariance statistic and the independence score and to
set the deviation parameter $x$.

\begin{proof}
First suppose $d_\beta<\beta\nu$.  Choose the fixed constant $b>0$ small
enough that $C\sqrt b\le\kappa/3$, where $C$ is the constant in
\eqref{eq:scoreconcentration} and $\kappa$ is the gap constant in
\eqref{eq:scoregap}.  Set $x=bn\beta^2\nu^2$.  By reducing the fixed upper
ranges $\beta_0,\nu_0$ if necessary, also arrange
$b\beta_0^2\nu_0^2\le c_0$, where $c_0$ is from
Proposition~\ref{prop:score-estimator}.  Thus, whenever
$n\beta^2\nu^2\ge 8/b$, the choice
$x=bn\beta^2\nu^2$ lies in $[8,c_0n]$.  Construct the corresponding
$x$-indexed estimators from
Proposition~\ref{prop:score-estimator}, and choose the direction with the
smaller value of $\widehat J_{d,x}$.
Then the estimation error in \eqref{eq:scoreconcentration} is at most one
third of the gap in \eqref{eq:scoregap}, except with probability
$Ce^{-cn\beta^2\nu^2}$.  If $n\beta^2\nu^2<8/b$, an arbitrary tie
break has risk at most one.  In this case, because $d_\beta<\beta\nu$,
\[
 n(d_\beta^2+\beta^2\nu^2)<16/b.
\]
Thus, after choosing the prefactor in \eqref{eq:upperrisk} at least
$\exp(16c/b)$, its right-hand side is at least one and also bounds this
small-sample branch.

If $d_\beta\ge\beta\nu$, use the covariance sign rule of
Proposition~\ref{prop:covariance}.  In the two cases the exponent is,
respectively, at least a constant times $n\beta^2\nu^2$ and $nd_\beta^2$.
Since
\[
 \max\{d_\beta^2,\beta^2\nu^2\}
 \le d_\beta^2+\beta^2\nu^2
 \le2\max\{d_\beta^2,\beta^2\nu^2\},
\]
this proves \eqref{eq:upperrisk}.
\end{proof}

Proposition~\ref{prop:upperrisk} implies the upper bound in
Theorem~\ref{thm:main}.  Write \eqref{eq:upperrisk} as
$A\exp(-an\xi)$, where $\xi=d_\beta^2+\beta^2\nu^2$, and choose
\[
 C_0\ge\frac1a\left(1+\frac{\log A}{\log(1/\delta_0)}\right).
\]
Then $n\ge C_0\log(1/\delta)/\xi$ implies
$A e^{-an\xi}\le A\delta^{aC_0}\le\delta$ for every
$\delta\le\delta_0$.  Thus it remains only to handle the integer sample size
\[
 n=\left\lceil
 C_0\frac{\log(1/\delta)}{\xi}
 \right\rceil
\]
appearing in the definition of $N_2^\star$.  Shrink the parameter ranges so
$\xi\le1$ and take $\delta_0\le e^{-1}$.  Then
\[
 \left\lceil C_0\frac{\log(1/\delta)}{\xi}\right\rceil
 \le (C_0+1)\frac{\log(1/\delta)}{\xi},
\]
which gives the stated sample-complexity upper bound.

\section{Matching two-point lower bound}
\label{sec:lower}

Proposition~\ref{prop:upperrisk} completes the upper bound.  To show that its
two signals cannot be improved, we now construct one admissible forward law
and one admissible reverse law separated by squared Hellinger distance at most
$C(d_\beta^2+\beta^2\nu^2)$.  The construction splits their discrepancy into
a rotation of size $\beta\nu$ and a nonorthogonal matrix perturbation of size
$d_\beta$.

\begin{proposition}[An admissible least-favorable forward--reverse pair]
\label{prop:hardpair}
There are $\beta_{\rm lb},\nu_{\rm lb}>0$ and $C<\infty$, depending only on
the fixed constants, such that, for every
$0<\beta\le\beta_{\rm lb}$ and $0<\nu\le\nu_{\rm lb}$, there exist
$P_{\rm f}\in\cP_{\dirf}(\beta,\nu)$ and
$P_{\rm r}\in\cP_{\dirr}(\beta,\nu)$ satisfying
\[
 H^2(P_{\rm f},P_{\rm r})
 \le C\{d_\beta^2+\beta^2\nu^2\}.
\tag{10.1}\label{eq:hardHellinger}
\]
Here, for densities $p,q$ with respect to a common dominating measure,
\[
 H^2(P,Q):=1-\int\sqrt{pq}
 =\frac12\norm{\sqrt p-\sqrt q}_2^2.
\]
\end{proposition}

\begin{proof}
Use the source path \eqref{eq:sourcepath} with
\[
 \eps:=\nu/A_\circ.
\tag{10.2}\label{eq:epsilonnu}
\]
By \eqref{eq:pathpsi}--\eqref{eq:pathNG}, this source belongs to
$\cQ(K,\nu)$ when $\nu\le\nu_{\rm src}$.  Let
\[
 \pi_\eps(z):=p_\eps(z_1)p_\eps(z_2)
\]
be the density of two independent copies; we use the same symbol for the
induced probability law when writing push-forwards and Hellinger distances.
Put
\[
 \zeta_\beta:=\max\{\rho,1-\beta^2\},
 \qquad d_\beta=\beta^2+\zeta_\beta-1.
\tag{10.3}\label{eq:zetabeta}
\]
Since a minimax lower bound requires only one forward--reverse pair, it
suffices to use the positive edge sign.  Define
\[
 A_{\rm f}:=\overline\sigma
 \begin{pmatrix}1&0\\ \beta&\sqrt{\zeta_\beta}\end{pmatrix},
 \qquad
 A_{\rm r}:=\overline\sigma
 \begin{pmatrix}\sqrt{\zeta_\beta}&\beta\\0&1\end{pmatrix}.
\tag{10.4}\label{eq:hardmatrices}
\]
For a measurable map $A$, write $A_\#P$ for the push-forward of $P$, and set
\[
 P_{\rm f}:=(A_{\rm f})_\#\pi_\eps,
 \qquad
 P_{\rm r}:=(A_{\rm r})_\#\pi_\eps.
\]
The law $P_{\rm f}$ is the forward model with
\[
 a=\beta,\qquad \sigma_1=\overline\sigma,\qquad
 \sigma_2=\overline\sigma\sqrt{\zeta_\beta};
\]
whereas $P_{\rm r}$ is the reverse model with
\[
 b=\beta,\qquad \tau_2=\overline\sigma,\qquad
 \tau_1=\overline\sigma\sqrt{\zeta_\beta}.
\]
Since $\zeta_\beta\in[\rho,1]$, all four scales satisfy the required bounds.
Hellinger distance is invariant under a common
invertible transformation.  After applying $A_{\rm r}^{-1}$, they become
$T_\#\pi_\eps$ and $\pi_\eps$, where
\[
 T:=A_{\rm r}^{-1}A_{\rm f}
 =\begin{pmatrix}
 (1-\beta^2)/\sqrt{\zeta_\beta}&-\beta\\
 \beta&\sqrt{\zeta_\beta}
 \end{pmatrix},
 \qquad \det T=1.
\tag{10.5}\label{eq:Tmatrix}
\]
For $\vartheta\in\R$, write
\[
 \mathcal R_\vartheta
 :=\begin{pmatrix}\cos\vartheta&-\sin\vartheta\\
                    \sin\vartheta&\cos\vartheta\end{pmatrix},
 \qquad \vartheta_\beta:=\arcsin\beta.
\]
Then
\[
 R_\beta:=\mathcal R_{\vartheta_\beta}
 =\begin{pmatrix}
 \sqrt{1-\beta^2}&-\beta\\
 \beta&\sqrt{1-\beta^2}
 \end{pmatrix}.
\tag{10.6}\label{eq:Rbeta}
\]
Put $q_\beta:=\sqrt{1-\beta^2}$.  Because
$\zeta_\beta=q_\beta^2+d_\beta$, $q_\beta\ge\sqrt3/2$, and
$\sqrt{\zeta_\beta}\ge\sqrt\rho$,
\[
 \abs{\sqrt{\zeta_\beta}-q_\beta}
 =\frac{d_\beta}{\sqrt{\zeta_\beta}+q_\beta}\le Cd_\beta,
 \qquad
 \abs{\frac{q_\beta^2}{\sqrt{\zeta_\beta}}-q_\beta}
 =\frac{q_\beta}{\sqrt{\zeta_\beta}}
   \abs{q_\beta-\sqrt{\zeta_\beta}}
 \le Cd_\beta.
\]
Comparing the four entries in \eqref{eq:Tmatrix} and \eqref{eq:Rbeta}
therefore gives
\[
 \norm{T-R_\beta}_F\le C d_\beta.
\tag{10.7}\label{eq:strainmatrix}
\]

Let $\phi_2(z):=\phi(z_1)\phi(z_2)$ be the standard bivariate Gaussian
density, and write $\pi_\eps(z)=\phi_2(z)r_\eps(z)$, where
\[
 r_\eps(z)=\{1+\eps h_\circ(z_1)\}
             \{1+\eps h_\circ(z_2)\}.
\tag{10.8}\label{eq:qepsilon}
\]
The density ratio $r_\eps$ lies in $[1/4,9/4]$, and
$\sup_{z\in\R^2}\norm{\nabla r_\eps(z)}_{\ell_2}\le C\abs\eps$.
Rotational invariance of $\phi_2$
and the elementary inequality
$\abs{\sqrt x-\sqrt y}\le C\abs{x-y}$ on $[1/4,9/4]$ imply
\begin{align}
 H^2(\pi_\eps,(R_\beta)_\#\pi_\eps)
 &\le C\int\phi_2(z)
 \abs{r_\eps(z)-r_\eps(R_\beta^{-1}z)}^2\dd z\notag\\
 &\le C\beta^2\eps^2.
\tag{10.9}\label{eq:rotationH}
\end{align}
For the second inequality, let
$J=\left(\begin{smallmatrix}0&-1\\1&0\end{smallmatrix}\right)$.  The ordinary
fundamental theorem of calculus and the preceding uniform gradient bound
give, for every $z\in\R^2$,
\begin{align*}
 \abs{r_\eps(z)-r_\eps(R_\beta^{-1}z)}
 &\le\int_0^{\abs{\vartheta_\beta}}
 \abs{\ip{\nabla r_\eps(\mathcal R_{-t}z),
                    J\mathcal R_{-t}z}}\dd t\\
 &\le C\abs\eps\,\abs{\vartheta_\beta}\norm{z}_{\ell_2}
 \le C\abs\eps\,\beta\norm{z}_{\ell_2}.
\end{align*}
Squaring and integrating this bound against $\phi_2$ proves
\eqref{eq:rotationH}.

It remains to control the nonorthogonal perturbation from $R_\beta$ to $T$.
Set $B:=T-R_\beta$ and
$S_t=R_\beta+tB$, $0\le t\le1$.  By \eqref{eq:strainmatrix} and
$d_\beta\le\beta^2$,
\[
 \norm{S_t-R_\beta}_{\rm op}\le Cd_\beta\le C\beta^2.
\]
Choose $\beta_{\rm lb}$ so that the last quantity is at most $1/2$.
Since $R_\beta$ is orthogonal, every singular value of $S_t$ then lies in
$[1/2,3/2]$; in particular, every $S_t$ is invertible and
$\norm{S_t}_{\rm op}+\norm{S_t^{-1}}_{\rm op}\le C$.  The square root
of the density of $(S_t)_\#\pi_\eps$ is
\[
 G_t(x)=\abs{\det S_t}^{-1/2}
 \sqrt{\pi_\eps(S_t^{-1}x)}.
\tag{10.10}\label{eq:Gt}
\]
Put $A_t=S_t^{-1}B$ and $y=S_t^{-1}x$.
Differentiation gives the exact identity
\[
 \dot G_t(x)=-\frac12G_t(x)
 \left\{\tr A_t+\ip{\nabla\log \pi_\eps(y),A_ty}\right\}.
\tag{10.11}\label{eq:Gtdot}
\]
After the change of variables $x=S_ty$,
\begin{align}
 \norm{\dot G_t}_2^2
 &=\frac14\int \pi_\eps(y)
 \left\{\tr A_t+\ip{\nabla\log \pi_\eps(y),A_ty}\right\}^2\dd y
 \le C\norm{B}_F^2.
\tag{10.12}\label{eq:Gtdotbound}
\end{align}
Indeed,
\[
 \partial_j\log \pi_\eps(y)
 =-y_j+\frac{\eps h_\circ'(y_j)}{1+\eps h_\circ(y_j)},
\]
whose magnitude is at most $\abs{y_j}+C$ uniformly over the permitted $\eps$; also
$\pi_\eps\le(9/4)\phi_2$ and
$\norm{A_t}_F\le C\norm{B}_F$.  Gaussian fourth
moments then prove \eqref{eq:Gtdotbound}.  For every $x$, the map
$t\mapsto G_t(x)$ is $C^1$ with derivative \eqref{eq:Gtdot}; moreover,
\eqref{eq:Gtdotbound} gives
$\int_0^1\norm{\dot G_t}_2\dd t\le C\norm{B}_F$.  Hence the pointwise identity
$G_1-G_0=\int_0^1\dot G_t\dd t$ and Minkowski's integral inequality imply
\[
 \norm{G_1-G_0}_2
 \le\int_0^1\norm{\dot G_t}_2\dd t
 \le C\norm{B}_F.
\]
Together with \eqref{eq:strainmatrix} and the definition of Hellinger
distance, this gives
\[
 H((R_\beta)_\#\pi_\eps,T_\#\pi_\eps)
 =\frac1{\sqrt2}\norm{G_1-G_0}_2
 \le Cd_\beta.
\tag{10.13}\label{eq:strainH}
\]

The Hellinger triangle inequality, \eqref{eq:rotationH}, and
\eqref{eq:strainH} now yield
\[
 H(\pi_\eps,T_\#\pi_\eps)
 \le C(d_\beta+\beta\abs\eps).
\]
Square this inequality, use \eqref{eq:epsilonnu}, and absorb the fixed
$A_\circ$ into the constant to obtain \eqref{eq:hardHellinger}.
\end{proof}

Proposition~\ref{prop:hardpair} supplies a forward--reverse pair with the
required one-sample Hellinger separation.  Tensorization of Hellinger
affinity now converts that pair into a lower bound for every $n$-sample
directional test.

\begin{proposition}[Testing consequence]
\label{prop:testinglower}
There are fixed $c,C>0$ such that every possibly randomized decision rule
satisfies
\[
 \mathcal R_n(\widehat d_n;\beta,\nu)
 \ge\frac14\exp\{-Cn(d_\beta^2+\beta^2\nu^2)\}
\tag{10.14}\label{eq:testinglower}
\]
throughout a sufficiently small fixed $(\beta,\nu)$ range.
\end{proposition}

\begin{proof}
Let $P_{\rm f},P_{\rm r}$ be the pair in
Proposition~\ref{prop:hardpair}, and put
$\xi=d_\beta^2+\beta^2\nu^2$.  By decreasing the fixed upper ranges, assume
$C\xi\le1/2$.  The affinity
$\Aff(P,Q)=\int\sqrt{\dd P\dd Q}$ then satisfies
\[
 \Aff(P_{\rm f},P_{\rm r})=1-H^2(P_{\rm f},P_{\rm r})
 \ge1-C\xi\ge e^{-2C\xi}.
\tag{10.15}\label{eq:affinityone}
\]
Affinity tensorizes, so the $n$-sample affinity is at least $e^{-2Cn\xi}$.
For any two laws,
\[
 \TV(P,Q)\le\sqrt{1-\Aff(P,Q)^2},
 \qquad
 1-\TV(P,Q)\ge\frac12\Aff(P,Q)^2.
\tag{10.16}\label{eq:TVaffinity}
\]
The sum of the two testing errors is at least $1-\TV$, and hence their
maximum is at least $\Aff^2/4$.  Applying this to the product laws proves
\eqref{eq:testinglower}, after changing the numerical constant in the
exponent.  Randomization is covered by adjoining the same independent random
seed to both experiments, which does not change affinity.
\end{proof}

\section{Proof of the main theorem and interpretation}

The preceding two sections have established matching exponential upper and
lower risk bounds.  It remains only to choose common small parameter ranges
and translate those risk bounds into the sample-complexity statement.

\begin{proof}[Proof of Theorem~\ref{thm:main}]
Choose $\beta_0$ no larger than every fixed small upper bound required in
Propositions~\ref{prop:upperrisk} and \ref{prop:hardpair}, and choose
$\nu_0\le\min\{\nu_{\rm src},\nu_{\rm lb}\}$.  The upper bound follows from
Proposition~\ref{prop:upperrisk} as explained after its proof.

For the lower bound, take the infimum over all decision rules in
Proposition~\ref{prop:testinglower}; no minimizer need exist.  This gives
\[
 \mathcal R_{2,n}^\star(\beta,\nu)
 \ge\frac14e^{-Cn(d_\beta^2+\beta^2\nu^2)}.
\]
Choose $\delta_0\le1/16$.  Then
for every $n$ such that $\mathcal R_{2,n}^\star\le\delta$,
$\log(1/(4\delta))\ge\frac12\log(1/\delta)$ and rearrangement gives
\[
 n\ge c\frac{\log(1/\delta)}{d_\beta^2+\beta^2\nu^2}.
\]
Taking the infimum over such sample sizes proves the lower bound and completes
the theorem.
\end{proof}

\paragraph{Coverage of degenerate moment patterns.}
The upper bound never assumes a density.  The smooth density
\eqref{eq:sourcepath} is used only to exhibit a least-favorable lower pair.
Accordingly, the lower bound proves non-improvability along one admissible
source direction; uniformity over all directions in the nonparametric source
class is supplied by the upper bound.
Nor can the theorem be reduced to a fixed low-order cumulant.  For example,
\[
 \Pp(Z=0)=\frac23,\qquad
 \Pp(Z=\sqrt3)=\Pp(Z=-\sqrt3)=\frac16
\]
has mean zero, variance one, skewness zero, fourth moment three, and
$\norm Z_{\psi_2}=\sqrt{3/\log4}<K_{\rm G}$.  Its characteristic function is
$f_Z(t)=2/3+(1/3)\cos(\sqrt3t)$, so direct Gaussian integration gives
\begin{align*}
 \NonG_w(Z)^2
 &=\frac12+\frac49e^{-3/4}+\frac1{18}e^{-3}
   -2\sqrt{\frac23}\left(\frac23+\frac13e^{-1/2}\right)
   +\frac1{\sqrt2}\\
 &>9.9\times10^{-4}.
\end{align*}
Thus this law belongs to $\cQ(K,\nu)$ for every
$\nu\le0.031$.  The Hermite argument in Theorem~\ref{thm:maxwell} treats all
orders simultaneously.

\paragraph{Why the scale phase is unavoidable.}
When $\underline\sigma<\overline\sigma$ and
$\beta^2\le1-\rho$, Proposition~\ref{prop:covariance} shows that the forward
and reverse covariance classes intersect.  The hard matrices
\eqref{eq:hardmatrices} realize this intersection with identical covariance
when $d_\beta=0$, so covariance then contains no information for
distinguishing that forward--reverse pair.  When the scale interval is a
singleton, covariance remains directional even at the Gaussian limit.
Together, the covariance separation calculation and the matching hard pair
show that $d_\beta$ is the parameter governing this transition.

\end{document}